\documentclass[final,hidelinks,onefignum,onetabnum]{siamart250211}

\usepackage{array}
\usepackage{amsmath,epsfig,comment}
\numberwithin{equation}{section}
\usepackage{tcolorbox}
\usepackage{amssymb}
\usepackage{pifont}
\usepackage{graphicx}
\usepackage{amstext}
\usepackage{mathrsfs}
\usepackage{multirow}
\usepackage{threeparttable}
\usepackage{subfigure}
\usepackage[multiple]{footmisc}
\usepackage{color}

\usepackage{amssymb,amsmath,cite}
\usepackage{epsfig}
\usepackage{color}
\usepackage{bm}
\usepackage{algorithmic}
\usepackage[ruled,vlined,linesnumbered]{algorithm2e}
\usepackage{graphicx}
\usepackage{subfigure}
\usepackage{multirow}
\usepackage{cite}
\usepackage{hyperref}
\usepackage{tikz}
\usetikzlibrary{arrows.meta,positioning}
\hypersetup{
	colorlinks=true,
	linkcolor=black,
	filecolor=magenta,
	urlcolor=cyan,
	citecolor=blue
}
\usepackage{extarrows}%

\makeatletter
\def\@thmcountersep{.}
\makeatother
\newtheorem{assumption}{Assumption}[section]{\bfseries}{\normalfont}
\DeclareMathOperator*{\argmin}{arg\,min}
\DeclareMathOperator*{\argmax}{arg\,max}

\usepackage{float}%
\usepackage{booktabs}%
\usepackage{multirow}%

\DeclareMathAlphabet\mathbfcal{OMS}{cmsy}{b}{n}

\newcommand{\rgrad}{\mathrm{grad}\,}

\newcommand{\prox}{\mathrm{prox}}
\newcommand{\proj}{\mathrm{proj}}
\usepackage{amsfonts}
\usepackage{graphicx}
\usepackage{epstopdf}
\usepackage{algorithmic}
\ifpdf
  \DeclareGraphicsExtensions{.eps,.pdf,.png,.jpg}
\else
  \DeclareGraphicsExtensions{.eps}
\fi

\newsiamremark{remark}{Remark}
\newsiamremark{hypothesis}{Hypothesis}
\crefname{hypothesis}{Hypothesis}{Hypotheses}
\newsiamthm{claim}{Claim}
\newsiamremark{fact}{Fact}
\crefname{fact}{Fact}{Facts}

\headers{RIEMANNIAN DESCENT--ASCENT WITH RECURSIVE MOMENTUM}{Meng Xu,
	Bo Jiang,
	Ya-Feng Liu, and Anthony Man-Cho So}

\title{A Stochastic Riemannian Alternating Descent Ascent Method for Nonsmooth Composite Expectation Optimization on Riemannian Manifolds\thanks{Submitted to the editors DATE.
}}

\author{Meng Xu\thanks{ICMSEC, Academy of Mathematics and Systems Science, Chinese Academy of Sciences, and University of Chinese Academy of Sciences, Beijing, China.
  (\email{xumeng22@mails.ucas.ac.cn}).}
\and Bo Jiang\thanks{Ministry of Education Key Laboratory of NSLSCS, School of Mathematical Sciences, Nanjing Normal University, Nanjing, China
  (\email{jiangbo@njnu.edu.cn}).}
    \and Ya-Feng Liu\thanks{Ministry of Education Key Laboratory of Mathematics and Information Networks, School of Mathematical Sciences, Beijing University of Posts
    	and Telecommunications, Beijing, China.
  	  	(\email{yafengliu@bupt.edu.cn}).}
    \and Anthony Man-Cho So\thanks{Department of Systems Engineering and Engineering Management, The Chinese University of Hong Kong, HKSAR, China.
  	  		(\email{manchoso@se.cuhk.edu.hk}).}
}

\usepackage{amsopn}

\ifpdf
\hypersetup{
  pdftitle={A Stochastic Riemannian Alternating Descent Ascent Method for Nonsmooth Composite Expectation Optimization on Riemannian Manifolds},
  pdfauthor={Meng Xu, Bo Jiang, Ya-Feng Liu, and Anthony Man-Cho So}
}
\fi

\begin{document}

\maketitle

\begin{abstract}
In this paper, we consider a class of Riemannian nonsmooth composite expectation optimization problems, which arises in various machine learning, signal processing, and statistics applications. Noting that these problems admit structured minimax reformulations, we propose an efficient algorithm, named stochastic Riemannian alternating descent ascent method with recursive momentum (StoRADA-RM), to tackle them. StoRADA-RM performs one or multiple Riemannian stochastic gradient descent steps and then a proximal gradient ascent step at each iteration. To compute the Riemannian stochastic gradient, we propose a vector transport-free recursive momentum estimator that requires only $\mathcal{O}(1)$ stochastic gradient evaluations per iteration. We prove that StoRADA-RM returns an  $\epsilon$-Riemannian-stochastic-stationary point of a given problem in the said class in $\mathcal{O}(\epsilon^{-3})$ iterations while making $\mathcal{O}(\epsilon^{-3})$ calls to a stochastic first-order oracle (SFO). Both the iteration complexity and SFO complexity bounds are the best known in the literature for the said class of problems. The latter even matches the optimal lower bound for smooth nonconvex optimization with stochastic first-order algorithms. We then present numerical results on sparse principal component analysis and coordinate-independent sparse estimation to demonstrate the superior performance of our proposed method.
\end{abstract}



\section{Introduction}\label{sec:introduction}
In this paper, we consider Riemannian nonsmooth expectation optimization problems of the form
\begin{equation}\label{prob:p1}
	\min_{x\in\mathcal{M}} \left\{\Phi(x):=f(x)+h(\mathcal{A}(x))\right\},\quad f(x):=\mathbb{E}_{\omega\sim\mathcal{D}}[f(x;\omega)],
\end{equation}
where $\omega$ is a random vector with probability distribution $\mathcal{D}$, $\mathcal{M}$ is a Riemannian manifold embedded in a finite-dimensional Euclidean space $\mathcal{E}_1$, $f(\cdot;\theta):\mathcal{E}_1\to\mathbb{R}$ is a continuously differentiable function for each $\theta$ in the range of $\omega$, $h:\mathcal{E}_2\to\mathbb{R}$ with $\mathcal{E}_2$ being another finite-dimensional Euclidean space is a convex and $L_h$-Lipschitz continuous function, and $\mathcal{A}:\mathcal{E}_1\to\mathcal{E}_2$ is a smooth map.  
When the range of $\omega$ is finite and given by $\{ \theta_1,\theta_2, \ldots,\theta_N \}$, and $\omega$ has the uniform distribution, the function $f$ reduces to the finite-sum form $f(x)=\tfrac{1}{N}\sum_{i=1}^{N}f(x;\theta_i)$. 
Many applications in machine learning, signal processing, and statistics give rise to instances of problem \eqref{prob:p1}, such as sparse principal component analysis (PCA) \cite{zou2018selective,jolliffe2003modified}, coordinate-independent sparse estimation (CISE) \cite{chen2010cordinate,xiao2021exact}, and robust low rank matrix completion \cite{huang2021robust}.

A variety of deterministic algorithms can be applied to solve problem \eqref{prob:p1} in the finite-sum setting, including Riemannian subgradient-type methods \cite{borckmans2014riemannian,hosseini2017riemannian,hosseini2018line,li2021weakly,hu2023constraint}, Riemannian proximal gradient-type methods \cite{chen2020proximal,huang2022riemannian,huang2023inexact,wang2022manifold,chen2024nonsmooth,liu2024penalty}, Riemannian smoothing-type algorithms \cite{beck2023dynamic,peng2023riemannian,zhang2023riemannian}, Riemannian splitting-type methods \cite{lai2014splitting,kovnatsky2016madmm,deng2023manifold,li2023riemannian,zhou2023semismooth,deng2024oracle,xu2025oracle}, and Riemannian minimax algorithms \cite{xu2023efficient2,xu2024riemannian}.
Despite the growing interest in developing stochastic algorithms for tackling problem \eqref{prob:p1}, existing works on this topic remain limited. In this paper, we focus on the case where the full gradient of $f$ is inaccessible as in the online setting, or intractable due to the extremely large number of component functions as in the finite-sum setting. Our goal is to develop an efficient stochastic algorithm for tackling problem \eqref{prob:p1}.

\subsection{Related Works}
When the nonsmooth term $h\equiv0$, problem \eqref{prob:p1} reduces to a smooth expectation optimization problem. A typical method for tackling such a problem is Riemannian stochastic gradient descent (SGD) \cite{bonnabel2013stochastic}, which is an extension of the classical Euclidean SGD method for the manifold setting. 
To further improve practical performance and theoretical guarantees, variance reduction techniques have been extensively studied in both Euclidean and Riemannian settings. 
Notable Riemannian stochastic algorithms with variance reduction include Riemannian stochastic variance-reduced gradient (RSVRG) methods \cite{zhang2016first,jiang2017vector,sato2019riemannian}, the Riemannian stochastic recursive gradient (RSRG) method \cite{kasai2018riemannian}, Riemannian stochastic path-integrated differential estimator (RSPIDER) methods \cite{zhang2018r,zhou2019faster}, and the
Riemannian stochastic recursive momentum (RSTORM) method \cite{han2021riemannian}.
It should be noted that RSVRG and RSRG are designed for finite-sum minimization problems, and they require computing the full gradient once in every fixed number of iterations. As such, they are not applicable to the setting of our interest, where full gradients are inaccessible. 
Although RSPIDER can avoid computing the full gradient and achieve the optimal $\mathcal{O}(\epsilon^{-3})$ stochastic first-order oracle (SFO) complexity for finding an $\epsilon$-Riemannian-stochastic-stationary ($\epsilon$-RSS) point of problem \eqref{prob:p1}, it still needs to compute a large-batch gradient of size $\mathcal{O}(\epsilon^{-2})$ in each epoch and a mini-batch size of $\mathcal{O}(\epsilon^{-1})$ at each inner iteration.
To address these limitations, Han and Gao \cite{han2021riemannian} proposed RSTORM, which only requires $\mathcal{O}(1)$ stochastic gradient evaluations per iteration and avoids computing full or large-batch gradients. 
They further established that RSTORM achieves a near-optimal $\widetilde{\mathcal{O}}(\epsilon^{-3})$ SFO complexity for finding  $\epsilon$-RSS points of smooth expectation optimization problems on Riemannian manifolds.\footnote{Here and in what follows, we use $\widetilde{\mathcal{O}}(\cdot)$ to hide poly-logarithmic factors.}

When the nonsmooth term $h \not= 0$, various efforts on developing stochastic algorithms for tackling problem \eqref{prob:p1} have focused on the Euclidean setting, where $\mathcal{M}\subseteq\mathcal{E}_1$ is a closed convex set. Of particular interest are stochastic proximal gradient methods with variance-reduced gradient estimators, which are motivated by the success of variance reduction techniques in improving both  numerical performance and SFO complexity. Specifically, when $\mathcal{M}=\mathcal{E}_1$ and $\mathcal{A}=\mathcal{I}$ with $\mathcal{I}$ being the identity map, Pham et al. proposed Prox-SARAH \cite{pham2020proxsarah} and Wang et al. proposed SpiderBoost \cite{wang2019spiderboost}, both achieving an SFO complexity of $\mathcal{O}(\epsilon^{-3})$ for finding an $\epsilon$-stochastic-stationary point of problem \eqref{prob:p1}. However, these methods require computing a large-batch gradient of size $\mathcal{O}(\epsilon^{-2})$ once in every fixed number of iterations and a mini-batch size of $\mathcal{O}(\epsilon^{-1})$ at other iterations, similar to their smooth counterparts.
Tran-Dinh et al. \cite{tran2022hybrid} proposed Hybrid-SGD and showed that it achieves the same $\mathcal{O}(\epsilon^{-3})$ complexity for finding an $\epsilon$-stochastic-stationary point of problem \eqref{prob:p1}, provided that the initial batch size is chosen as $\mathcal{O}(\epsilon^{-1})$. Without this large initial batch, however, Hybrid-SGD can only guarantee a near-optimal $\widetilde{\mathcal{O}}(\epsilon^{-3})$ complexity. More recently, Xu and Xu \cite{xu2023momentum} proposed ProxSTORM, which incorporates the variance-reduced gradient estimator from STORM \cite{cutkosky2019momentum} to avoid computing large-batch gradients and achieves an SFO complexity of $\mathcal{O}(\epsilon^{-3})$ for finding  an $\epsilon$-stochastic-stationary point of problem \eqref{prob:p1}.

Recently, there has been growing interest in developing stochastic algorithms for tackling problem \eqref{prob:p1} when $h\not=0$ and $\mathcal{M}$ is a compact submanifold of $\mathcal{E}_1$. For the setting where $\mathcal{A}=\mathcal{I}$, Li et al. \cite{li2021weakly} proposed a Riemannian stochastic subgradient (RSSG) method when $h$ is a weakly convex function and showed that it has an iteration complexity of $\mathcal{O}(\epsilon^{-4})$ for driving a certain stationarity measure of problem \eqref{prob:p1} to below $\epsilon$.
Wang et al. \cite{wang2022riemannian} proposed two Riemannian stochastic proximal gradient methods, R-ProxSGD and R-ProxSPB, when $h$ is a convex function. These two methods are the Riemannian counterparts of the Euclidean proximal SGD \cite{rosasco2020convergence} and proximal SpiderBoost \cite{wang2019spiderboost} methods, respectively. 
To obtain an $\epsilon$-RSS point of problem \eqref{prob:p1}, R-ProxSGD and R-ProxSPB require $\mathcal{O}(\epsilon^{-4})$ and $\mathcal{O}(\epsilon^{-3})$ SFO calls, respectively.
For the setting where $\mathcal{A}$ is a linear map, Peng et al. \cite{peng2023riemannian} proposed a Riemannian stochastic smoothing method when $h$ is a weakly convex function, which has an SFO complexity of $\mathcal{O}(\epsilon^{-5})$ for finding a so-called generalized $\epsilon$-stationary point of problem \eqref{prob:p1}.
Using a splitting technique, Deng et al. \cite{deng2024oracle} proposed a stochastic manifold inexact augmented Lagrangian (StoManIAL) framework when $h$ is a convex function.
By employing the Riemannian recursive momentum method as a subroutine, StoManIAL achieves an SFO complexity of $\widetilde{\mathcal{O}}(\epsilon^{-3.5})$ for finding an $\epsilon$-RSS point of problem \eqref{prob:p1}.
Although R-ProxSPB and StoManIAL achieve good complexity results, both of them are nested-loop algorithms that require solving a subproblem at each iteration. 
Specifically, R-ProxSPB requires solving each subproblem exactly, while StoManIAL requires either finding an $\epsilon_k$-RSS point of the AL subproblem (which is difficult to verify) or performing $\mathcal{O}(2^k)$   Riemannian stochastic gradient steps (which may result in unnecessary but expensive computation in practice) at the $k$-th iteration. 

In view of the above developments, it is naturally of interest to develop simple yet efficient \textit{single-loop} stochastic algorithms for tackling problem \eqref{prob:p1} that achieve the optimal SFO complexity.
Very recently, Deng et al. \cite{deng2025single} proposed a single-loop Riemannian stochastic smoothing method (RSSM), which makes use of the variance-reduced gradient estimator of RSTORM \cite{han2021riemannian}, and showed that it has an SFO complexity of $\mathcal{O}(\epsilon^{-3})$ for finding an $\epsilon$-RSS point of problem \eqref{prob:p1}. 
Despite their theoretical guarantees, the numerical performance of smoothing-based methods is highly sensitive to the choices of the smoothing parameter and stepsizes and can be unsatisfactory even in the deterministic setting. Moreover, RSSM employs a Riemannian recursive momentum estimator that combines tangent vectors at different iterates. As such, it requires vector transport operations, which introduces additional computational cost.

\begin{table}[t]
	\centering
	\footnotesize
	\renewcommand{\arraystretch}{1.5}
	\tabcolsep=0.2cm
	\caption{Summary of stochastic algorithms with SFO complexity results for tackling problem \eqref{prob:p1}. Here, “LB” denotes the largest required batch size; “wcvx” and “cvx” denote weakly convex and convex, respectively; “P-Free” represents that the algorithmic parameters are independent of problem parameters (e.g., Lipschitz constants); and “VT-Free” means that the algorithm does not require vector transport operations. 
	}
	\begin{tabular}{c|c|c|c|c|c|c|c}
		\hline
		Algorithm & $\mathcal{A}$& $h$ & LB& P-Free & VT-Free & Complexity& Loop\\ 
		\hline
		RSSG \cite{li2021weakly}& $\mathcal{I}$ & wcvx &$\mathcal{O}(1)$ & \ding{51} & \ding{51}&  $\mathcal{O}(\epsilon^{-4})$ & single \\
		\hline
		R-ProxSGD \cite{wang2022riemannian}& $\mathcal{I}$ & cvx& $\mathcal{O}(\epsilon^{-2})$ & \ding{55} & \ding{51} & $\mathcal{O}(\epsilon^{-4})$& nested\\
		\hline
		R-ProxSPB \cite{wang2022riemannian}& $\mathcal{I}$ & cvx& $\mathcal{O}(\epsilon^{-2})$ & \ding{55} & \ding{55} & $\mathcal{O}(\epsilon^{-3})$& nested\\
		\hline
		StoManIAL \cite{deng2024oracle}& linear& cvx & $\mathcal{O}(1)$ &  \ding{55} & \ding{55} &$\mathcal{O}(\epsilon^{-3.5})$& nested \\
		\hline
		RSSM \cite{deng2025single} & nonlinear& wcvx & $\mathcal{O}(1)$ &  \ding{51} & \ding{55} &$\mathcal{O}(\epsilon^{-3})$& single\\
		\hline
		StoRADA-RM & nonlinear& cvx & $\mathcal{O}(1)$ &  \ding{51} & \ding{51} &$\mathcal{O}(\epsilon^{-3})$& single  \\
		\hline \end{tabular}
	\label{tab: complexity comparision}
\end{table}

\subsection{Our Contributions}
We summarize the main contributions of this paper as follows.
\begin{itemize}
	\item We reformulate problem \eqref{prob:p1} as a Riemannian minimax problem and propose an efficient stochastic Riemannian alternating descent ascent method with recursive momentum (StoRADA-RM) to tackle it. 
	At each iteration, the proposed method performs one or multiple Riemannian stochastic gradient descent steps to approximately minimize a value function associated with the inner maximization problem, followed by a proximal gradient ascent step on a regularized version of the objective function of the minimax problem.
	Since full gradient information is often inaccessible or too expensive to compute in our setting, we adopt a Euclidean momentum-based variance-reduced gradient estimator that requires only $\mathcal{O}(1)$ SFO calls per iteration. With such an estimator, our proposed method does not need to perform any vector transport operations, which is in sharp contrast to the method proposed in \cite{deng2025single}.
	Our proposed StoRADA-RM is obtained by adapting our earlier RADA-RGD \cite{xu2024riemannian}, which is developed for a deterministic setting, to a stochastic setting.
	Still, even when specialized to the setting where $\mathcal{M}$ is a closed convex set in a Euclidean space, StoRADA-RM remains new.
	\item %
	We establish the convergence of StoRADA-RM and show that both its iteration complexity and SFO complexity for finding an $\epsilon$-RSS point of problem \eqref{prob:p1} are $\mathcal{O}(\epsilon^{-3})$. The iteration complexity matches the best known in the literature for deterministic Riemannian nonsmooth composite optimization; see, e.g., \cite{beck2023dynamic,deng2024oracle,xu2025oracle}. Remarkably, the SFO complexity of StoRADA-RM not only matches the state-of-the-art in the literature for Riemannian nonsmooth composite expectation optimization (see, e.g., \cite{wang2022riemannian,deng2025single}) but also attains the optimal lower bound established in \cite{arjevani2023lower} for smooth nonconvex optimization with stochastic first-order algorithms. 
	To obtain the said complexity results, we develop new techniques to (i) control the estimation error of the recursive momentum estimator in the presence of a nonsmooth composite term in the objective function, thus extending the results in \cite{levy2021storm}; and (ii) analyze the progress made at each iteration of our proposed method in the absence of a sufficient decrease property of the iterates generated by the Riemannian stochastic gradient steps, thus bypassing a key requirement in the analysis of RADA-RGD in \cite{xu2024riemannian}.
	\item %
	We present numerical results on sparse PCA and CISE problems to demonstrate the efficiency of our proposed StoRADA-RM. The proposed method consistently returns solutions with lower objective values and has significantly lower computational cost compared with existing state-of-the-art methods, including R-ProxSPB \cite{wang2022riemannian}, StoManIAL \cite{deng2024oracle}, and RSSM \cite{deng2025single}. In particular, although both StoRADA-RM and RSSM are single-loop algorithms with the same SFO complexity of $\mathcal{O}(\epsilon^{-3})$, StoRADA-RM is substantially more efficient in practice.  
\end{itemize}

A comparison of existing methods for tackling problem \eqref{prob:p1} is summarized in Table \ref{tab: complexity comparision}.

\subsection{Organization}

The rest of the paper is organized as follows. 
We first introduce the notation and cover some preliminaries on Riemannian optimization and convex analysis in Section \ref{sec:preliminaries}. 
We then present our proposed algorithm in Section \ref{sec:proposed algorithm framework}, followed by its convergence analysis in Section \ref{sec: iteration complexity}.
Then, we present numerical results on sparse PCA and CISE to illustrate the efficiency of the proposed algorithms in Section \ref{sec: numerical results}.
Finally, we draw some conclusions in Section \ref{sec: concluding remarks}.

\section{Notation and Preliminaries}\label{sec:preliminaries}

We first introduce the notation and some basic concepts in Riemannian optimization \cite{absil2008optimization,boumal2023introduction}. 
We use $\langle \,\!\cdot\,, \cdot\rangle$ and $\| \cdot \|$ to denote the standard inner product and its induced norm on the Euclidean space $\mathcal{E}$, respectively. 
For a subset $\mathcal{X} \subseteq \mathcal{E}$ and a point $y \in \mathcal{E}$, we define the distance from $y $ to $\mathcal{X}$ as $\mathrm{dist}(y,\,\mathcal{X}) := \inf_{x\in\mathcal{X}} \|y-x\|$, and we denote by $\mathrm{conv}\,\mathcal{X}$ the convex hull of $\mathcal{X}$.
For a linear map $\mathcal{A}:\mathcal{E}_1\to\mathcal{E}_2$, its adjoint is denoted by $\mathcal{A}^\top:\mathcal{E}_2\to\mathcal{E}_1$.
Let $\mathcal{M}$ be a submanifold embedded in $\mathcal{E}$ and $ \mathrm{T}_{x}\mathcal{M} $ denote the tangent space to $\mathcal{M}$ at $x\in\mathcal{M}$.
Throughout this paper, we take the standard inner product $\langle \cdot,\cdot \rangle$ on the Euclidean space $\mathcal{E}$ as the Riemannian metric on $\mathcal{M}$. Then, for a smooth function $f: \mathcal{E} \rightarrow \mathbb{R}$ and a point $x \in \mathcal{M}$, the Riemannian gradient of $f$ at $x$ is given by
$\rgrad f(x)=\proj_{\mathrm{T}_{x}\mathcal{M}}(\nabla f(x))$, where $\proj_{\mathcal{X}}(\cdot)$ is the Euclidean projection operator onto a closed set $\mathcal{X}$ and $\nabla f(x)$ is the Euclidean gradient of $f$ at $x$.
A retraction at $x \in \mathcal{M}$ is a smooth map $\mathrm{R}_x: \mathrm{T}_x \mathcal{M} \to \mathcal{M}$ satisfying (i) $\mathrm{R}_x(\mathbf{0}_x) = x$, where $\mathbf{0}_x$ is the zero element in $\mathrm{T}_x \mathcal{M}$; (ii) $\frac{\mathrm{d}}{\mathrm{d} t} \mathrm{R}_x (t v)|_{t = 0} = v$ for all $v \in \mathrm{T}_x \mathcal{M}$. 
Throughout this paper, we assume that the retraction is globally defined, namely, $\mathrm{R}_x(v)$ is well defined for every $x \in \mathcal{M}$ and $v \in \mathrm{T}_x \mathcal{M}$. 
For a sampling set $S = \{ \theta_{1},\,\theta_{2},\ldots,\,\theta_{b} \}$ with cardinality $|S|=b$, where each $\theta_{i}$ is sampled independently according to the distribution $\mathcal{D}$, we define the stochastic Euclidean gradient as 
\begin{equation*}
	\nabla f_{S}(x) := \frac{1}{|S|} \sum_{i=1}^{|S|} \nabla f(x; \theta_i). 
\end{equation*}

Now, we review some basic notions in convex analysis \cite{rockafellar2009variational,beck2017first}. 
Let $g:\mathcal{E}\to(-\infty,+\infty]$ be a proper closed convex function whose domain is given by $\mathrm{dom}\ g:= \{x \in \mathcal{E} \mid g(x) < +\infty\}$. Its conjugate function is defined as
$%
g^*(y):=\sup_{x\in\mathcal{E}}\left\{\langle y,\,x\rangle-g(x)\right\}.
$
For a given constant $\lambda > 0$, the proximal map and the Moreau envelope 
of $g$ are defined as 
\begin{equation*}
\prox_{ \lambda g}(x)=\argmin_{u\in\mathcal{E}}\left\{g(u)+\frac{1}{2\lambda}\|x-u\|^2\right\}
\end{equation*} 
and
\begin{equation*}\label{Moreau envelope}
	M_{{\lambda g}}(x)=\min_{u\in\mathcal{E}}\left\{g(u)+\frac{1}{2\lambda}\|x-u\|^2\right\},
\end{equation*}
respectively. 
For any $x \in \mathcal{E}$, we have the following so-called Moreau decomposition:
\begin{equation}\label{Moreau Decomposition}
x=\prox_{\lambda g}(x)+\lambda\prox_{{g^*/\lambda}}\left(\frac{x}{\lambda}\right).
\end{equation}
The following theorem characterizes the  smoothness of the Moreau envelope.%
\begin{theorem}\label{gradient of Moreau}\text{\hspace{-0.05cm}(\hspace{-0.01cm}\cite[Theorems 6.60 and 6.39 ]{beck2017first})}
Let $g:\mathcal{E}\to(-\infty,+\infty]$ be a proper closed convex function and $\lambda>0$. Then, for any $x\in\mathcal{E}$,
\begin{equation*}\label{gradient of Moreau envelope}
	\nabla M_{\lambda g}(x)=\frac{1}{\lambda}(x-\prox_{\lambda g}(x))=\mathrm{prox}_{g^*/\lambda}\left(\frac{x}{\lambda}\right)\in\partial g(\prox_{\lambda g}(x)),
\end{equation*}
where $\partial g$ is the subdifferential of $g$. 
\end{theorem}
Moreover, the following theorem shows the equivalence between Lipschitz continuity and boundedness of the domain of the conjugate of a convex function.
\begin{theorem}\label{equi between L }\text{(\hspace{-0.01cm}\cite[Theorem 4.23]{beck2017first})}
Let $g:\mathcal{E}\to\mathbb{R}$ be a convex function. Then, $g$ is $L$-Lipschitz continuous (i.e.,  $|g(x) - g(x')| \leq L \|x - x'\|$ for all $x, x' \in \mathcal{E}$) if and only if $\mathrm{sup}_{y \in \mathrm{dom}\,g^*} \|y\| \leq L$.
\end{theorem}

\section{Proposed StoRADA-RM}\label{sec:proposed algorithm framework}

In this section, we propose an efficient algorithm for tackling the Riemannian nonsmooth expectation optimization problem \eqref{prob:p1}. 

\paragraph{Main framework}
First, we reformulate problem \eqref{prob:p1} as the Riemannian minimax problem
\begin{equation}\label{equ: minimax}
\min_{x \in \mathcal{M}}\,\max_{y\in\mathcal{E}_2}\,\left\{F(x,y):=f(x)+\langle\mathcal{A}(x),y\rangle-h^*(y)\right\}.
\end{equation}
In the original formulation \eqref{prob:p1}, the composite term $h(\mathcal{A}(\cdot))$ is nonsmooth and nonconvex, which makes the problem difficult to handle directly. 
By contrast, the reformulation \eqref{equ: minimax} decouples the nonsmoothness and nonconvexity: Its objective function is smooth but nonconvex in $x$ and concave but nonsmooth in $y$. One approach to tackling the Riemannian minimax problem \eqref{equ: minimax} is to apply the algorithmic framework RADA that we developed in \cite{xu2024riemannian}. Specifically, at the $k$-th iteration, where $k \ge 1$, we approximately solve the following problem to compute the next iterate $x_{k+1}$:
\begin{equation}\label{potential Phi_k}
\min_{x \in \mathcal{M}}\left\{\Phi_{k}(x):= \max_{y\in\mathcal{E}_2}\left\{F(x,y)-\frac{\lambda_k}{2}\|y\|^2-\frac{\beta_{k}}{2}\|y-y_{k}\|^2\right\}\right\}.
\end{equation}
Here, $\lambda_k > 0$ is the regularization parameter, $y_{k} \in \mathrm{dom}\, h^*$ is the proximal center, and $\beta_{k} \geq 0$ is the proximal parameter. We compute the update
\begin{equation}
	\begin{aligned}\label{AlgorithmupdateY}
		y_{k+1}&=\argmax_{y\in\mathcal{E}_2}\left\{F(x_{k+1},y)-\frac{\lambda_k}{2}\|y\|^2-\frac{\beta_{k}}{2}\|y-y_{k}\|^2\right\}\\
		&=\prox_{{h^*/(\lambda_k + \beta_k)}}\left(\frac{\mathcal{A}(x_{k+1}) + \beta_k y_k}{\lambda_k + \beta_k}\right),
	\end{aligned}	
\end{equation} 
 which can be regarded as performing a proximal gradient ascent step on the regularized function $F(x_{k+1},\cdot)-\frac{\lambda_k}{2}\|\cdot\|^2$.
 
\paragraph{Riemannian stochastic gradient descent}
To approximately solve the subproblem \eqref{potential Phi_k} for a given $k \ge 1$, we leverage  the differentiability of value function $\Phi_{k}$. Specifically, observe that $\Phi_{k}$ can be written as
\begin{equation}\label{Moreau Phi h} 
	\begin{aligned}
		\Phi_{k}(x)
		= &\,f(x) + M_{{(\lambda_k+\beta_k)h}}\left(\mathcal{A}(x)+\beta_{k}y_{k}\right)-\frac{\beta_{k}}{2}\|y_k\|^2.
	\end{aligned}
\end{equation}
As shown in \cite[Section 3.1]{xu2024riemannian}, the function $\Phi_k$ is differentiable with
\begin{equation}\label{nabla Phik}
	\nabla\Phi_{k}(x)=\nabla f(x)+\nabla\mathcal{A}(x)^\top \prox_{{h^*}/(\lambda_k+\beta_{k})}\left(\frac{\mathcal{A}(x) + \beta_k y_k}{\lambda_k + \beta_k}\right),
\end{equation}
where $\nabla\mathcal{A}$ denotes the Jacobian of $\mathcal{A}$. 
Unlike in the setting considered in \cite{xu2024riemannian}, the full gradient $\nabla f$ in \eqref{nabla Phik} is intractable in our current setting. Thus, instead of performing Riemannian gradient steps on the value function $\Phi_k$ as in \cite[Algorithm 3]{xu2024riemannian}, we perform Riemannian stochastic gradient steps. Specifically, let $T \ge 1$ be the number of Riemannian stochastic gradient steps to be performed on $\Phi_k$.
Let $x_{k,1}=x_k$ and $x_{k,t}$ be the iterate at the $t$-th inner iteration for $t\in\{1,2,\ldots,T\}$. In view of \eqref{nabla Phik}, we define the Riemannian stochastic gradient of $\Phi_k$ at $x_{k,t}$ as 
\begin{equation}\label{equ:Dkt}
	D_{k,t}^R  = \proj_{\mathrm{T}_{x_{k,t}}\mathcal{M}}\left(d_{k,t}+ \nabla \mathcal{A}(x_{k,t})^\top \prox_{{h^*}/(\lambda_k+\beta_{k})}\left(\frac{\mathcal{A}(x_{k,t}) + \beta_k y_k}{\lambda_k + \beta_k}\right)\right).
\end{equation}
Here, $d_{k,t}$ is a stochastic estimator of $\nabla f(x_{k,t})$, which will be specified in \eqref{equ: RM estimator}.
We then choose a stepsize $\zeta_{k,t}>0$ and compute the update
\begin{equation}\label{equ:update:xkt+1}
	x_{k,t+1} = \mathrm{R}_{x_{k,t}}(-\zeta_{k,t}D_{k,t}^R)
\end{equation}
for $t=1,2,\ldots,T$  to obtain $x_{k+1}=x_{k+1,1}=x_{k,T+1}$.

\paragraph{Recursive momentum estimator} 
To construct the stochastic estimator $d_{k,t}$ for $t=1,2,\ldots,T$, we adopt the recursive momentum approach proposed in \cite{cutkosky2019momentum,levy2021storm}.
The approach relies on the following standard assumption.
\begin{assumption}
The function $f$ is equipped with an SFO, which is a map that, when queried at the point $x \in \mathcal{M}$, returns an unbiased estimate $f(x;\theta)$ of the function value $f(x)$ and an unbiased estimate $\nabla f(x;\theta)$ of the gradient $\nabla f(x)$.
\end{assumption}
Given the initial stochastic estimator  $d_{k,1} $ of $\nabla f(x_{k,1})$,  the stochastic estimator $d_{k,t+1}$ for $t=1,2,\ldots,T$  is constructed via 
\begin{equation}\label{equ: RM estimator}
	d_{k,t+1}=\nabla f_{S_{k,t+1}}(x_{k,t+1})+(1-\alpha_{k,t})(d_{k,t}-\nabla f_{S_{k,t+1}}(x_{k,t})).
\end{equation}
Here, $\alpha_{k,t} \in (0,1]$ is a given weight
and $S_{k,t+1}$ is a sampling set of size $|S_{k,t+1}|=b$, where $b \ge 1$ is given. We assume that the initial sampling set $S_1$ and the mini-batches $\{S_{k, t + 1}\mid k \geq 1, 1 \le t \le T\}$ are mutually independent. By convention, we set $d_{k,T+1} = d_{k+1,1}$.

\paragraph{Choice of parameters}
To obtain an efficient instantiation of the above algorithmic framework, we need to choose the parameters $\{\lambda_k\}$, $\{\beta_k\}$, $\{ \alpha_{k,t} \}$, and $\{ \zeta_{k,t} \}$ judiciously, so as to simultaneously control the errors of the stochastic estimators $\{ d_{k,t} \}$ and ensure sufficient progress towards stationarity is made when the value functions $\{ \Phi_k \}$ in successive iterations are approximately minimized. First, for $k \ge 1$ and $t = 1,2,\ldots,T$, we set the momentum weight $\alpha_{k,t}$ and stepsize $\zeta_{k,t}$ as
\begin{equation}\label{equ:alpha}
	\alpha_{k,t}=((k-1)T+t)^{-2/3}
\end{equation}
and
\begin{equation}\label{equ: zetakt}
	\zeta_{k,t}=\xi_k\alpha_{k,t}^{1/3}\left({{\color{black}{\delta_0+}}}\sum_{i=1}^{k-1}\sum_{j=1}^{T}\|D_{i,j}^R\|^2+\sum_{j=1}^{t}\|D_{k,j}^R\|^2\right)^{-1/3},
\end{equation}
respectively, where $\delta_0 > 0$ and  $\xi_k\in[\xi_{\min},\xi_{k-1}]$ with $\xi_0\geq\xi_{\min}>0$ to ensure that the stepsize is nonzero even when all the Riemannian stochastic gradients are zero. These choices are motivated by the parameter balancing mechanism in STORM+ \cite{levy2021storm,deng2025single}. The underlying idea is to coordinate the use of past gradient information with the movement of the iterates. On one hand, the momentum weight $\alpha_{k,t}$ is used to induce a variance reduction effect through the correction term $(1-\alpha_{k,t}) (d_{k,t} - \nabla f_{S_{k,t+1}} (x_{k,t}))$ in \eqref{equ: RM estimator}; cf. \cite{levy2021storm}. To accommodate the $T$ inner updates in StoRADA-RM, we index $\alpha_{k,t}$ by the total number of inner updates, i.e., $(k-1)T+t$. On the other hand, the stepsize $\zeta_{k,t}$ uses the momentum weight $\alpha_{k,t}$ and the Riemannian stochastic gradients $\{D_{k',t'}^R\}$ computed so far to control the progress made between two successive outer iterations. 

Next, let $\lambda_1 > 0$, $\beta_1 \ge 0$, and $\rho > 1$ be fixed. For $k \ge 1$, we set the regularization parameter $\lambda_k$ and the proximal parameter $\beta_k$ in the subproblem \eqref{potential Phi_k} as
\begin{equation}\label{equ:lambdak:betak}
\lambda_k=\frac{\lambda_1}{k^{1/3}}
\quad\mbox{and}\quad
0\leq\beta_{k}\leq\frac{\beta_1}{(k)^\rho},
\end{equation}
respectively. Unlike in the deterministic setting considered in \cite{xu2024riemannian}, where $\{ \lambda_k \}$ is chosen to be a constant sequence, we choose a decreasing sequence $\{ \lambda_k \}$ whose rate of decay is compatible with that of $\{ \zeta_{k,t} \}$. Moreover, we impose a summability condition on $\{ \beta_k \}$. These choices allow us to bound the accumulated change in values of the different value functions that arise in the course of our proposed method and show that in expectation, both the norm of the Riemannian gradient of the value function and the regularization error decay at the same rate; see Theorem \ref{theo: complexity}.

The proposed StoRADA-RM method is formally presented in Algorithm \ref{Algorithm: RADASTORM}.

\begin{algorithm}[H]\label{Algorithm: RADASTORM}%
\fontsize{10pt}{\baselineskip}\selectfont
\caption{StoRADA-RM for solving problem \eqref{prob:p1}}

\textbf{Input:} $x_{1}\in \mathcal{M}$, $y_{1}\in \mathrm{dom}\, h^*$, $\lambda_{1}>0$, $\beta_{1}\geq 0$, $\rho>1$, $\xi_{0}\geq \xi_{\min}>0$, $\delta_0>0$, integers $b, K, T \ge 1$.

Generate the sampling set $S_{1}$ with $|S_{1}|=b$ and compute $d_1 = \nabla f_{S_{1}}(x_1)$.

\For{$k=1,2,\ldots,K$}
{ 
	
	Let $x_{k,1}=x_k$ and $d_{k,1} = d_k$. 
	
	Set ${\lambda_k=\lambda_1/k^{1/3}}$ and $\xi_k\in[\xi_{\min},\xi_{k-1}]$.
	
	\For{$t=1,2,\ldots,T$}
	{
		Compute  $D_{k,t}^R$ as \eqref{equ:Dkt}, set $\alpha_{k,t}$ and $\zeta_{k,t}$ as in \eqref{equ:alpha} and \eqref{equ: zetakt}, respectively, and update $x_{k, t+1}$ as \eqref{equ:update:xkt+1}. 
		
		Generate the sampling set $S_{k,t+1}$ with $|S_{k,t+1}|=b$ and update $d_{k,t+1}$ as  \eqref{equ: RM estimator}.
	}
	
	Update $x_{k+1}=x_{k,T+1}$ and $d_{k+1} = d_{k,T+1}$.
	
	Calculate $y_{k+1}$ via \eqref{AlgorithmupdateY}.
	
	Update $0\leq\beta_{k+1}\leq \beta_{1}/(k+1)^\rho$.   
	
}

Return $x_{\hat{k}}$ uniformly from $\{x_{k}\}_{k=1}^K$.
\end{algorithm}

Several remarks on StoRADA-RM are in order. First, although the method is a natural instantiation of the RADA framework in the stochastic setting, the use of the recursive momentum estimator only yields an approximation of the full gradient of the value function. The approximation error at each inner iteration accumulates, which leads to difficulty in establishing the sufficient decrease property needed in the analysis of RADA (see \cite[displayed eqn. (13)]{xu2024riemannian}). 
To circumvent this difficulty, we need to choose the momentum weights, the stepsizes, and the regularization parameters carefully so that the gradient estimation errors and the progress made between two successive outer iterations can be controlled simultaneously. 
This presents a significant challenge in the analysis of StoRADA-RM that is not seen in the analyses of RADA \cite{xu2024riemannian} and the recursive momentum estimator-based methods in \cite{cutkosky2019momentum,levy2021storm}.

Second, StoRADA-RM is simple to implement while retaining the optimal SFO complexity.
It can be regarded as a single-loop method since $T$ is fixed (and can even be set to one). By contrast, the stochastic Riemannian AL method in \cite{deng2024oracle} requires solving a subproblem either to  $\epsilon_k$-stochastic stationarity or by performing $\mathcal{O}(2^k)$ inner updates at the $k$-th iteration, both of which are computationally prohibitive.
On the other hand, the batch size $b$ in StoRADA-RM can be taken as $\mathcal{O}(1)$. This makes StoRADA-RM easier to implement  and potentially more efficient than R-ProxSPB \cite{wang2022riemannian}, as the latter requires  computing the full gradient or a large-batch gradient of size $\mathcal{O}(\epsilon^{-2})$ once in every fixed number of iterations.
Moreover, the stepsizes $\{\zeta_{k,t}\}$ are independent of the problem parameters, such as the Lipschitz constant of $\nabla f$. This makes the stepsize scheme of StoRADA-RM more flexible than those of not only the stochastic R-ProxSPB \cite{wang2022riemannian} and StoManIAL \cite{deng2024oracle} but also the deterministic RADA-RGD \cite{xu2024riemannian}.

Finally, we compare StoRADA-RM with the recently proposed RSSM \cite{deng2025single}, which is also  single-loop and employs the recursive momentum estimator. Specifically, RSSM performs a single Riemannian stochastic gradient step to approximately solve the smoothed subproblem
\begin{equation*}\label{Rsmooth}
\min_{x\in\mathcal{M}} \left\{ f(x)+M_{\lambda_k h}(\mathcal{A}(x))\right\},
\end{equation*}
which can be viewed as a special case of our subproblem \eqref{potential Phi_k} with $\beta_k \equiv 0$ for all $k \ge 1$. 
Despite this connection, the two methods differ fundamentally in their design philosophies. StoRADA-RM exploits the intrinsic minimax structure of the original problem to decouple the nonsmoothness and nonconvexity. Moreover, it allows multiple Riemannian stochastic gradient steps. Although such a flexibility complicates the convergence analysis of the method, it yields more accurate solutions to the subproblems and contributes to the favorable numerical performance of the method; see Section \ref{sec: numerical results}. 
Lastly, StoRADA-RM adopts the Euclidean recursive momentum estimator \eqref{equ: RM estimator}, which, unlike RSSM, avoids the need to perform vector transport operations, thereby further improving its computational efficiency.

	\section{Complexity Analysis of StoRADA-RM}\label{sec: iteration complexity}
	In  this section, we study the convergence behavior of  StoRADA-RM (Algorithm \ref{Algorithm: RADASTORM}). %
	\subsection{Assumptions and Useful Lemmas} 
	We begin with the assumptions used in our analysis.
	\begin{assumption}\label{assumption: level bound}
		The manifold $\mathcal{M} \subseteq \mathcal{E}_1$ is compact.%
	\end{assumption}
	\begin{assumption}\label{assumption1}
		\begin{itemize}
			\item[(i)] The function $f: \mathcal{E}_1 \to \mathbb{R}$ is continuously differentiable and satisfies the descent property
			\begin{equation*}\label{descent Euclidean}
				f(x') \leq f(x) + \langle \nabla f(x),\,x' - x\rangle + \frac{L_f}{2} \|x' - x\|^2, \quad \forall\,x,x' \in \mathcal{M}. 
			\end{equation*}
			\item[(ii)] The map $\mathcal{A}: \mathcal{E}_1 \to \mathcal{E}_2$ and its Jacobian $\nabla \mathcal{A}$ are $L_{\mathcal{A}}^0$-Lipschitz and $L_{\mathcal{A}}^1$-Lipschitz over $\mathrm{conv}\,\mathcal{M}$, respectively. In other words, for any $x,x' \in \mathrm{conv}\,\mathcal{M}$, we have
			\begin{subequations}
				\begin{align}
					\| \mathcal{A}(x) - \mathcal{A}(x') \| \leq L_{\mathcal{A}}^0 \|x - x'\|,\nonumber\\
					\| \nabla \mathcal{A}(x) - \nabla \mathcal{A}(x') \| \leq L_{\mathcal{A}}^1 \|x - x'\|\nonumber. 
				\end{align}
			\end{subequations}
			Moreover, the Jacobian $\nabla \mathcal{A}$ is bounded over $\mathrm{conv}\,\mathcal{M}$, i.e., 
			\begin{equation*}\label{rhoA}
				\rho_{\mathcal{A}}: = \sup_{x\,\in\,\mathrm{conv}\,\mathcal{M}}\, \|\nabla \mathcal{A} (x)\| < + \infty.
			\end{equation*} 
		\end{itemize}
	\end{assumption}
	\begin{assumption}\label{assumption: stochastic}
		The unbiased gradient estimate $\nabla f(x;\theta)$ has bounded variance and satisfies a so-called mean-squared smoothness property, i.e., there exist $\sigma > 0$ and $\tilde{L}_f > 0$ such that for all $x,\,x'\in\mathcal{M}$, we have
		\begin{align}
			&\mathbb{E}_{\omega \sim \mathcal{D}}[\nabla f(x;\omega)] =\nabla f(x),\label{equ: gradient is unbiased}\\
			&\mathbb{E}_{\omega \sim \mathcal{D}}[\|\nabla f(x;\omega)-\nabla f(x)\|^2]\leq\sigma^2\label{ineq: bounded var},\\
			&\mathbb{E}_{\omega \sim \mathcal{D}}[\|\nabla f(x;\omega)-\nabla f(x';\omega)\|^2]\leq \tilde{L}_f^2\|x-x'\|^2. \quad\label{inequ: tilde Lf}
		\end{align}
	\end{assumption}
	Assumptions \ref{assumption: level bound} and \ref{assumption1} are standard in the context of Riemannian nonsmooth composite optimization; see, e.g., \cite{chen2024nonsmooth,chen2020proximal,huang2022riemannian,huang2023inexact,xu2025oracle,deng2025single}. Assumption \ref{assumption: stochastic} is standard in the context of Riemannian stochastic optimization; see, e.g., \cite{han2021riemannian,deng2025single,wang2022riemannian,xu2023momentum}.
	
	The following lemma, which is extracted from \cite[Appendix B]{boumal2019global}, shows that the retraction satisfies first- and second-order boundedness conditions on $\mathcal{M}$. 
	\begin{lemma}\label{lemma Bound retraction}
		Suppose that Assumption \ref{assumption: level bound} holds. Then, there exist  $\kappa_1, \kappa_2>0$ such that 
		\begin{equation*}\label{bound retraction}
			\|\mathrm{R}_x(v)-x\|\leq\kappa_1\|v\| \quad \text{{and}} \quad \|\mathrm{R}_x(v)-x-v\|\leq\kappa_2\|v\|^2
		\end{equation*} 
		for any $x\in\mathcal{M},\,v\in\mathrm{T}_{x}\mathcal{M}$.
	\end{lemma}
	
	Based on Lemma \ref{lemma Bound retraction}, we have the following descent property of the value function $\Phi_k$ for $k \ge 1$. 
	\begin{lemma}(\cite[Lemma 3]{xu2024riemannian}\label{lemma l-smooth E})
		Suppose that Assumption \ref{assumption1} holds. Then, for all $k \ge 1$, the value function $\Phi_k$ in \eqref{potential Phi_k} satisfies the following properties: 
		\begin{itemize}
			\item[(i)] {\bf Euclidean descent.} For any $x,x' \in \mathcal{M}$, we have
			\begin{equation*}\label{lemma l-smooth E ineq E}
				{\Phi_k}(x') \leq {\Phi_k}(x) + \langle \nabla {\Phi_k}(x),\,x' - x \rangle + \frac{\ell_k}{2} \|x' - x\|^2, %
			\end{equation*}
			where 
			\begin{equation}\label{ell_k}
				\ell_k= L_f+L_hL_{\mathcal{A}}^1+\frac{{\rho_\mathcal{A}}L_\mathcal{A}^0}{\lambda_k+\beta_{k}}.
			\end{equation}
			\item[(ii)] 
			{\bf Riemannian descent.} For any $x\in\mathcal{M}$ and $v\in\mathrm{T}_x\mathcal{M}$, we have
			\begin{equation}\label{lemma l-smooth E ineq R}
				{\Phi_k}(\mathrm{R}_x(v))\leq {\Phi_k}(x)+\langle \rgrad {\Phi_k}(x),\,v\rangle + \frac{L_k}{2}\|v\|^2,
			\end{equation} 
			where 	
			\begin{align} \label{equ:Lk} 
				L_k &= \ell_k \kappa_1^2 +2 M_G\kappa_2 \quad \text{and}\quad {M_G=\max_{x\in\mathcal{M}} \|\nabla f(x)\|+\rho_\mathcal{A} L_h}.
			\end{align}
		\end{itemize}
	\end{lemma}
	From \eqref{equ:lambdak:betak}, \eqref{ell_k} and \eqref{equ:Lk}, we have the upper bound 
	\begin{equation}\label{inequ: L_k<Lc}
		L_k\leq \frac{L_c}{\lambda_k+\beta_{k}},
	\end{equation}
	where
$
		L_c=\rho_{\mathcal{A}}L_\mathcal{A}^0\kappa_1^2+(\lambda_{1}+\beta_{1})((L_f+L_hL_\mathcal{A}^1)\kappa_1^2+2M_G\kappa_2).
$
	Lastly, we need the following technical lemma, which concerns the behavior of nonzero nonnegative vectors.
	\begin{lemma}\label{lemma: storm+}\text{(\hspace{-0.01cm}\cite[Lemma 3]{levy2021storm})}
		Let $a_1>0, a_2, \ldots, a_n \geq 0$ and $p \in (0, 1)$. Then,
		\begin{align}
			\sum_{i=1}^{n}\frac{a_i}{\left(\sum_{j=1}^{i}a_j\right)^p}\leq\frac{1}{1-p}\left(\sum_{i=1}^{n}a_i\right)^{1-p}.\nonumber
		\end{align}
	\end{lemma}
	It follows from Lemma \ref{lemma: storm+} that, for any $\delta_0>0$, $a_1,\ldots,a_n\geq0$, and $p\in(0,1)$,
	\begin{equation}\label{inequ: shifted storm+}
		\sum_{i=1}^{n}\frac{a_i}{\left(\delta_0+\sum_{j=1}^{i}a_j\right)^p}
		\leq\frac{1}{1-p}\left(\sum_{i=1}^{n}a_i\right)^{1-p}.
	\end{equation}

\subsection{Preliminary Estimates}	
To set the stage for our convergence analysis of Algorithm \ref{Algorithm: RADASTORM}, we define
	\begin{equation}\label{equ: Upsilon gds}
	\begin{aligned}
		\overline{\Phi}:=\sup_{x\in\mathcal{M}}\Phi(x),\quad &\underline{\Phi}:=\inf_{x\in\mathcal{M}}\Phi(x),\quad e_{k,t} := d_{k,t}-\nabla  f(x_{k,t}),\\
		\Upsilon_{g}:= \sum_{k=1}^{K}\sum_{t=1}^{T}\|\rgrad\Phi_{k}(x_{k,t})\|^2,\quad&\Upsilon_{d}:=\sum_{k=1}^{K}\sum_{t=1}^{T}\|D_{k,t}^R\|^2,
		\quad\Upsilon_{e}:=\sum_{k=1}^{K}\sum_{t=1}^{T}\|e_{k,t}\|^2.
	\end{aligned}
	\end{equation}
	We begin with the following estimates of the value functions $\{ \Phi_k \}$.
	\begin{lemma}\label{lemma: Phi+1-Phi}
		Let $\{x_{k}\}$ be the sequence generated by Algorithm \ref{Algorithm: RADASTORM}. Then, we have
		\begin{equation}\label{suffcient decrease pf 1}
			\Phi_{k+1}(x_{k+1})-\Phi_{k}(x_{k+1})\leq\left(2\beta_{k}+\frac{\lambda_k-\lambda_{k+1}}{2}\right)L_h^2, \quad \forall k \ge 1.
		\end{equation} 
		Moreover, for any $x\in\mathcal{M}$, we have
		\begin{equation}\label{ineq: Phik geq Phi}
			\Phi_{k}(x)\geq\underline{\Phi}-\frac{\lambda_1}{2}L_h^2-2\beta_{1}L_h^2, \quad \forall k \ge 1.
		\end{equation}
	\end{lemma}
	\begin{proof}
			For simplicity, denote
		\begin{align}
			F_k(x,y):=&\,F(x,y)-\frac{\lambda_k}{2}\|y\|^2-\frac{\beta_{k}}{2}\|y-y_k\|^2,\label{Fk}\\
			y_{k+\frac12}:=&\,\prox_{h^*/(\lambda_k+\beta_k)}\left(\frac{\mathcal{A}(x_k)+\beta_{k}y_k}{\lambda_k+\beta_{k}}\right),\label{yk+1/2}
		\end{align}
		where $k \ge 1$. By the definition of the proximal operator, we know that $y_k,\,y_{k+\frac12}\in\mathrm{dom}\,h^*$ for all $k\geq1$. Since $h$ is assumed to be convex and $L_h$-Lipschitz continuous, it follows from Theorem \ref{equi between L } that
		\begin{align}\label{inequ: yk<Lh}
			\|y_k\|\leq L_h\quad \text{and} \quad \|y_{k+\frac12}\|\leq L_h, \quad \forall\, k\geq1.
		\end{align}
		Then, we have
		\begin{align}
			{}\Phi_{k+1}(x_{k+1})-\Phi_{k}(x_{k+1})
			={}& F_{k+1}(x_{k+1},y_{k+\frac{3}{2}})-F_k(x_{k+1},y_{k+1})\nonumber\\
			\overset{(\text{a})}{\leq}{}& F_{k+1}(x_{k+1}, y_{k + \frac32}) - F_k(x_{k+1}, y_{k + \frac32})\nonumber \\
			\overset{(\text{b})}{=}{}&F(x_{k+1},y_{k+\frac{3}{2}})-\frac{\lambda_{k+1}}{2}\|y_{k+\frac{3}{2}}\|^2-\frac{\beta_{k+1}}{2}\|y_{k+\frac{3}{2}}-y_{k+1}\|^2\nonumber\\
			{}&-F(x_{k+1},y_{k+\frac{3}{2}})+\frac{\lambda_k}{2}\|y_{k+\frac{3}{2}}\|^2+\frac{\beta_{k}}{2}\|y_{k+\frac{3}{2}}-y_{k}\|^2\nonumber\\
			\overset{(\text{c})}{\leq}{} &\,\left(2\beta_{k}+\frac{\lambda_k-\lambda_{k+1}}{2}\right)L_h^2,\nonumber
		\end{align}
		where $y_{k+\frac32}$ is a shorthand for $y_{(k+1)+\frac12}$, (a) comes from the optimality of $y_{k+1}$ in \eqref{AlgorithmupdateY}, (b) holds by the definition of $F_{k}$ in \eqref{Fk}, and (c) is due to \eqref{inequ: yk<Lh}.
		
		On the other hand, for any $x \in \mathcal{M}$, we deduce from \eqref{potential Phi_k}  that 
		\begin{equation*}
			\begin{aligned}
				\Phi_{k}(x)&=\max_{y\in\mathcal{E}_2}\left\{f(x)+\langle \mathcal{A}(x),y\rangle-h^*(y)-\frac{\lambda_k}{2}\|y\|^2-\frac{\beta_{k}}{2}\|y-y_k\|^2\right\}\\
				&\overset{\text{(a)}}{\geq}\max_{y\in\mathcal{E}_2}\left\{f(x)+\langle\mathcal{A}(x),y\rangle-h^*(y)-\frac{\lambda_k}{2}L_h^2-2\beta_{k}L_h^2\right\} \\
				&= f(x)+h(\mathcal{A}(x))-\frac{\lambda_k}{2}L_h^2-2\beta_{k}L_h^2\geq\underline{\Phi}-\frac{\lambda_1}{2}L_h^2-2\beta_{1}L_h^2,
			\end{aligned}
		\end{equation*}
		where (a) is due to the fact that $y \in \mathrm{dom}\,h^*$ and Theorem \ref{equi between L },
		and the last inequality follows from the definition of $\Phi$ and $\underline{\Phi}$ in \eqref{prob:p1} and \eqref{equ: Upsilon gds}, respectively, together with the parameter bounds in \eqref{equ:lambdak:betak}. 
	\end{proof}
    	
    Based on Lemma \ref{lemma: Phi+1-Phi}, we now prove the following lemma, which shows how the accumulated squared Riemannian gradient norm $\Upsilon_g$ associated with the value functions $\{\Phi_k\}$ can be controlled by the accumulated squared Riemannian stochastic gradient norm $\Upsilon_d$ and the accumulated squared gradient estimation error $\Upsilon_e$. 
	\begin{lemma}\label{lemma: pre decent of Phi}
		Suppose that Assumptions \ref{assumption: level bound} and \ref{assumption1} hold. Let $\{x_{k,t}\}$ be the sequence generated by Algorithm \ref{Algorithm: RADASTORM}. Then, we have
		\begin{align}\label{inequ: pre decent of Phi}
			\Upsilon_{g}
			\overset{}{\leq}&\,\frac{2(\Upsilon+\Delta_\Phi)}{\xi_{\min}}(TK)^{2/9}{{\color{black}{(\delta_0+\Upsilon_{d})^{1/3}}}}
			+ \frac{3\xi_{0}L_c}{2\lambda_{1}}K^{1/9}\Upsilon_{d}^{2/3}+\Upsilon_{e},
		\end{align}
		where 
		\begin{equation}\label{equ: MPhi and Upsilon}
			\begin{aligned}
				\Delta_\Phi=\,\overline{\Phi}-\underline{\Phi}+\frac{\lambda_1L_h^2}{2}+2\beta_{1}L_h^2 \quad\text{and}\quad \Upsilon=\,{{\color{black}{2\max\{L_h,L_h^2\}}}}\sum_{k=1}^{\infty}\beta_{k}+\frac{\lambda_{1}L_h^2}{2}.
			\end{aligned}
		\end{equation}
	\end{lemma}
	\begin{proof}
		We start by analyzing the descent property at each inner iteration of Algorithm \ref{Algorithm: RADASTORM} to obtain a bound on $\Upsilon_g$.
		Let $k \in \{1,2,\ldots,K\}$ and $t \in \{1,2,\ldots,T\}$ be fixed.
		Using $x_{k,t+1} = \mathrm{R}_{x_{k,t}}(-\zeta_{k,t}D_{k,t}^R)$ and \eqref{lemma l-smooth E ineq R}, we get
		\begin{align}\label{inequ: descent Phikt:1}
			\Phi_k(x_{k,t+1})- \Phi_{k}(x_{k,t})  
			\leq  - \zeta_{k,t}\langle \rgrad\Phi_{k}(x_{k,t}), D_{k,t}^R \rangle + \frac{\zeta_{k,t}^2 L_k}{2} \| D_{k,t}^R \|^2.
		\end{align}
		Note that
		\begin{align}
			-2\langle \rgrad\Phi_{k}(x_{k,t}), D_{k,t}^R \rangle \leq&\, \| D_{k,t}^R - \text{grad}\,\Phi_{k}(x_{k,t}) \|^2   - \| \text{grad}\,\Phi_{k}(x_{k,t}) \|^2\nonumber\\
			 \leq &\, \| e_{k,t} \|^2  - \| \text{grad}\,\Phi_{k}(x_{k,t}) \|^2,\nonumber
		\end{align}
		where the last inequality follows from the expression of $D^R_{k,t}$ in \eqref{equ:Dkt}, the expression of $\nabla \Phi_k(x_{k,t})$ in \eqref{nabla Phik}, the relation $\rgrad \Phi_k(x_{k,t}) = \proj_{\mathrm{T}_{x_{k,t}}\mathcal{M}}(\nabla \Phi_k(x_{k,t}))$, the definition of $e_{k,t}$ in \eqref{equ: Upsilon gds}, and the nonexpansiveness of the projection operator.
		Substituting this bound into \eqref{inequ: descent Phikt:1}, applying \eqref{inequ: L_k<Lc}, and 
		rearranging terms, we obtain
		\begin{align}\label{inequ: rgrad 1}
			\|\rgrad\Phi_{k}(x_{k,t})\|^2\leq\frac{2}{\zeta_{k,t}}(\Phi_{k}(x_{k,t})-\Phi_{k}(x_{k,t+1}))+\|e_{k,t}\|^2+\frac{L_c\zeta_{k,t}}{\lambda_k+\beta_{k}}\|D_{k,t}^R\|^2.
		\end{align}
		Denote $\Delta_{k,t} :=(\Phi_{k}(x_{k,t})-\Phi_{k}(x_{k,t+1}))/\zeta_{k,t}$. Summing \eqref{inequ: rgrad 1} over $k$ and $t$ and recalling \eqref{equ: Upsilon gds}, we have
		\begin{equation}
			\label{equ:Upsilon:g}
			\Upsilon_g \leq 
			{{\color{black}{2}}}\sum_{k = 1}^K \sum_{t = 1}^T \Delta_{k,t} + L_c \sum_{k = 1}^K \sum_{t = 1}^T \frac{\zeta_{k,t}}{\lambda_k + \beta_k} \|D_{k,t}^R\|^2 + \Upsilon_{e}. 
		\end{equation}
		
		Next, we bound the first term on the right-hand side, i.e., $\sum_{k=1}^K\sum_{t=1}^T \Delta_{k,t}$. Observe that for $k=1,2,\ldots,K$,
		\begin{align}\label{inequ: sum delta t}
			\sum_{t=1}^{T}\Delta_{k,t}
			=\,&\frac{\Phi_{k}(x_{k,1})}{\zeta_{k,1}}-\frac{\Phi_{k}(x_{k,T+1})}{\zeta_{k,T}}+\sum_{t=1}^{T-1}\left(\frac{1}{\zeta_{k,t+1}}-\frac{1}{\zeta_{k,t}}\right)\Phi_{k}(x_{k,t+1})\nonumber\\
			\leq\,&\frac{\Phi_{k}(x_{k,1})}{\zeta_{k,1}}-\frac{\Phi_{k}(x_{k,T+1})}{\zeta_{k,T}}+\left(\frac{1}{\zeta_{k,T}}-\frac{1}{\zeta_{k,1}}\right)\overline{\Phi},
		\end{align}
		where the last inequality is due to $\zeta_{k,t+1}\leq\zeta_{k,t}$ for $t=1,2,\ldots,T$ and 
		\begin{equation}\label{inequ: Phik<M}
			\Phi_{k}(x)\leq\Phi(x)\leq \overline{\Phi}, \quad \forall\,x\in\mathcal{M}.
		\end{equation}
 It follows from \eqref{inequ: sum delta t} that
		\begin{equation}\label{inequ: sum Deltakt kt 1}
		\begin{aligned}
			\sum_{k=1}^{K}\sum_{t=1}^{T}\Delta_{k,t}\leq&\,\sum_{k=1}^{K}\left(\frac{\Phi_{k}(x_{k,1})}{\zeta_{k,1}}-\frac{\Phi_{k}(x_{k,T+1})}{\zeta_{k,T}}+\left(\frac{1}{\zeta_{k,T}}-\frac{1}{\zeta_{k,1}}\right)\overline{\Phi}\right)\\
			=&\,\sum_{k=1}^{K-1}\left(\frac{\Phi_{k+1}(x_{k+1,1})}{\zeta_{k+1,1}}-\frac{\Phi_{k}(x_{k,T+1})}{\zeta_{k,T}}\right)+\sum_{k=1}^{K}\left(\frac{1}{\zeta_{k,T}}-\frac{1}{\zeta_{k,1}}\right)\overline{\Phi}\\
			&\,+\frac{\Phi_1(x_{1,1})}{\zeta_{1,1}}-\frac{\Phi_{K}(x_{K,T+1})}{\zeta_{K,T}}.
		\end{aligned}
	\end{equation}
		Note that 
		\begin{align}
			&\,\sum_{k=1}^{K-1}\left(\frac{\Phi_{k+1}(x_{k+1,1})}{\zeta_{k+1,1}}-\frac{\Phi_{k}(x_{k,T+1})}{\zeta_{k,T}}\right)\nonumber\\
			=&\,\sum_{k=1}^{K-1}\left(\frac{\Phi_{k+1}(x_{k+1,1})}{\zeta_{k+1,1}}-\frac{\Phi_{k}(x_{k,T+1})}{\zeta_{k+1,1}}+\frac{\Phi_{k}(x_{k,T+1})}{\zeta_{k+1,1}}-\frac{\Phi_{k}(x_{k,T+1})}{\zeta_{k,T}}\right)\nonumber\\
			\leq&\,\frac{\Upsilon}{\zeta_{K,1}}+\sum_{k=1}^{K-1}\left(\frac{1}{\zeta_{k+1,1}}-\frac{1}{\zeta_{k,T}}\right)\overline{\Phi},			\nonumber
		\end{align}
		where the last inequality is due to the convention $x_{k+1}=x_{k+1,1}=x_{k,T+1}$; the bound \eqref{suffcient decrease pf 1}; the definitions of $\Upsilon$ and $\overline{\Phi}$ in \eqref{equ: MPhi and Upsilon} and  \eqref{inequ: Phik<M}, respectively; and the properties $\zeta_{k+1,1}\leq\zeta_{k,T}$ and $\xi_k\leq\xi_{k-1}$.
		Substituting the above into \eqref{inequ: sum Deltakt kt 1}, we have 
		\begin{align}\label{inequ: sum Deltskt kt2}
			\sum_{k=1}^{K}\sum_{t=1}^{T}\Delta_{k,t}\leq&\,\frac{\Upsilon}{\zeta_{K,1}}+\frac{\overline{\Phi}}{\zeta_{K,T}}-\frac{\overline{\Phi}}{\zeta_{1,1}}+\frac{\Phi_1(x_{1,1})}{\zeta_{1,1}}-\frac{\Phi_{K}(x_{K,T+1})}{\zeta_{K,T}}{\leq} \frac{\Upsilon+\Delta_\Phi}{\zeta_{K,T}},
		\end{align}
		where the last inequality holds due to \eqref{inequ: Phik<M}, the definition of $\Delta_\Phi$ in \eqref{equ: MPhi and Upsilon}, and the lower bound on $\Phi_{k}$ in \eqref{ineq: Phik geq Phi}. 
		
		Moreover,  by the definitions of $\alpha_{k,t}$ and $\zeta_{k,t}$ in \eqref{equ:alpha} and \eqref{equ: zetakt}, respectively, we know that
		\begin{align}\label{inequ: sum D2}
			\sum_{k=1}^{K}\sum_{t=1}^{T}\frac{\zeta_{k,t}\|D^R_{k,t}\|^2}{\lambda_k+\beta_{k}}\overset{\text{(a)}}{\leq}&\,\frac{\xi_{0}}{\lambda_{1}}K^{1/9}\sum_{k=1}^{K}\sum_{t=1}^{T}\frac{\|D_{k,t}^R\|^2}{{{\color{black}{(\delta_0+\sum_{i=1}^{k-1}\sum_{j=1}^{T}\|D_{i,j}^R\|^2+\sum_{j=1}^{t}\|D_{k,j}^R\|^2)^{1/3}}}}}\nonumber\\
			\overset{\text{(b)}}{\leq}&\,\frac{3\xi_{0}}{2\lambda_{1}}K^{1/9}\Upsilon_{d}^{2/3},
		\end{align} 
		where (a) uses $\alpha_{k,t}^{1/3}/(\lambda_k+\beta_{k})\leq k^{-2/9}/(\lambda_{1}k^{-1/3})=k^{1/9}/\lambda_1\leq K^{1/9}/\lambda_{1}$ and $\xi_k \leq \xi_0$;  (b) follows from {{\color{black}{\eqref{inequ: shifted storm+} with $p=1/3$}}} and the definition of $\Upsilon_{d}$ in \eqref{equ: Upsilon gds}. 
		Plugging \eqref{inequ: sum Deltskt kt2} and \eqref{inequ: sum D2} into \eqref{equ:Upsilon:g}, and using the fact that \[
		{{\color{black}{\frac{1}{\zeta_{K,T}} \leq \frac{1}{\xi_{\min}} (TK)^{2/9} (\delta_0+\Upsilon_d)^{1/3}}}},
		\] which follows from \eqref{equ: zetakt} and \eqref{equ: Upsilon gds}, we obtain the desired inequality \eqref{inequ: pre decent of Phi}.
	\end{proof}
	
	\subsection{Bound on $\Upsilon_e$}
	Next, we  bound the accumulated gradient estimation error $\Upsilon_{e}$. To do this, we introduce the following lemma, which controls the estimation error of the recursive momentum in a single update.
		\begin{lemma}\label{lemma: bound estimator error}
			Suppose that Assumptions {{\color{black}{\ref{assumption: level bound},}}} \ref{assumption1}, and \ref{assumption: stochastic} hold. Let $\{x_{k,t}\}$ be the sequence generated by Algorithm \ref{Algorithm: RADASTORM}. Then, for all $k \in \{1,2,\ldots,K\}$ and $t \in \{1,2,\ldots,T\}$, we have 
			\begin{align}\label{inequ: bound estimator error}
				\mathbb{E}[\|e_{k,t+1}\|^2]\leq&\,(1-\alpha_{k,t}) \mathbb{E}[\| e_{k,t} \|^2]+\frac{2}{b}\alpha_{k,t}^2\sigma^2+\frac{2}{b}\tilde{L}_f^2\kappa_1^2{{\color{black}{\mathbb{E}[\zeta_{k,t}^2\|D_{k,t}^R\|^2]}}},
			\end{align}
			 where the expectation $\mathbb{E}[\cdot]$ is taken over all randomness in the course of Algorithm \ref{Algorithm: RADASTORM}.
		\end{lemma}
		\begin{proof}
			Let $k \in \{1,2,\ldots,K\}$ be fixed. For $t=1,\ldots,T$, let
			\[
			\mathcal{F}_{k,t+1}:=\{S_1\}\cup\{S_{i,j+1}\mid 1\leq i\leq k-1,\ 1\leq j\leq T\}\cup\{S_{k,j+1}\mid 1\leq j\leq t-1\}
			\]
			be the sampling history before $S_{k,t+1}$ is generated. By construction, $S_{k,t+1}$ is independent of $\mathcal{F}_{k,t+1}$ and
			\[
			\mathbb{E}[\| e_{k,t+1} \|^2] = \mathbb{E}[\mathbb{E}[ \| e_{k,t+1} \|^2 | \mathcal{F}_{k,t+1} ]].
			\]
			For simplicity, we omit the subscript “$k$" in the following proof.
			Recalling the definition of $e_{t}$ in \eqref{equ: Upsilon gds} and using \eqref{equ: RM estimator}, we have
			\begin{align}\label{inequ: bound skt 1}
				&\,\mathbb{E}[ \| e_{t+1} \|^2 |\mathcal{F}_{t+1} ] \nonumber\\
				= &\,\mathbb{E}[ \| (1 -\alpha_{t})  (d_{t} - \nabla f_{S_{t+1}}(x_{t})) + \nabla f_{S_{t+1}}(x_{t+1}) -\nabla f(x_{t+1}) \|^2 | \mathcal{F}_{t+1} ] \nonumber\\
				= &\,\mathbb{E}[ \| (1-\alpha_{t})  e_{t} + \alpha_{t} (\nabla f_{S_{t+1}}(x_{t+1})  -\nabla f(x_{t+1}) ) \nonumber\\
				&+ (1-\alpha_{t}) (  \nabla f_{S_{t+1}}(x_{t+1}) -  \nabla f_{S_{t+1}}(x_{t}) +  \nabla f(x_{t}) -\nabla f(x_{t+1}) ) \|^2 | \mathcal{F}_{t+1}] \nonumber\\
				\overset{(\text{a})}{=} &\,(1-\alpha_{t})^2 \| e_{t}\|^2 + \mathbb{E}[ \| \alpha_{t} (\nabla f_{S_{t+1}}(x_{t+1})  -\nabla f(x_{t+1}) ) \nonumber\\
				&+ (1-\alpha_{t}) (  \nabla f_{S_{t+1}}(x_{t+1}) -  \nabla f_{S_{t+1}}(x_{t}) +  \nabla f(x_{t}) -\nabla f(x_{t+1}) )   \|^2 | \mathcal{F}_{t+1}] \nonumber\\
				\overset{(\text{b})}{=}&\,(1-\alpha_{t})^2 \| e_{t} \|^2 + \frac{1}{b} \mathbb{E}_{\omega \sim \mathcal{D}}[ \| \alpha_{t} (\nabla f(x_{t+1};\omega)  -\nabla f(x_{t+1}) ) \nonumber\\
				&+ (1-\alpha_{t}) (  \nabla f(x_{t+1};\omega) -  \nabla f(x_{t};\omega) +  \nabla f(x_{t}) -\nabla f(x_{t+1}) )   \|^2 ] \nonumber\\
				\overset{\text{(c)}}{\leq}&\,(1-\alpha_{t})^2 \| e_{t} \|^2 + \frac{2}{b}\alpha_{t}^2 \mathbb{E}_{\omega \sim \mathcal{D}}[ \| \nabla f(x_{t+1};\omega)  -\nabla f(x_{t+1}) \|^2] \nonumber\\
				&+ \frac{2}{b}(1-\alpha_{t})^2\mathbb{E}_{\omega \sim \mathcal{D}}[ \|  \nabla f(x_{t+1};\omega) -  \nabla f(x_{t};\omega) +  \nabla f(x_{t}) -\nabla f(x_{t+1}) \|^2],
			\end{align}
			where (a) and (b) are due to \eqref{equ: gradient is unbiased}, and (c) uses the fact that $\|u + v\|^2 \leq 2 \|u\|^2 + 2\|v\|^2$ for any $u, v \in \mathbb{R}^n$. %
			Since $\mathbb{E}[\|z - \mathbb{E}[z] \|^2] \leq \mathbb{E}[\|z\|^2]$ for any random variable $z$, it follows that
			\begin{align}
				&\,\mathbb{E}_{\omega \sim \mathcal{D}}[ \|  \nabla f(x_{t+1};\omega) -  \nabla f(x_{t};\omega) +  \nabla f(x_{t}) -\nabla f(x_{t+1}) \|^2]\nonumber\\
				\leq&\,\mathbb{E}_{\omega \sim \mathcal{D}}[ \|  \nabla f(x_{t+1};\omega) -  \nabla f(x_{t};\omega) \|^2] 
				\leq\tilde{L}_f^2\kappa_1^2\zeta_{k,t}^2\|D_t^R\|^2,\nonumber
			\end{align}
			where the second inequality is from \eqref{inequ: tilde Lf} and Lemma \ref{lemma Bound retraction}. 
			Then, substituting the above into \eqref{inequ: bound skt 1} and using \eqref{ineq: bounded var} lead to 
			\begin{align}
				\mathbb{E}[ \| e_{t+1} \|^2 |\mathcal{F}_{t+1} ]\leq(1-\alpha_{t})^2 \| e_{t} \|^2+\frac{2}{b}\alpha_{t}^2\sigma^2+\frac{2}{b}(1-\alpha_{t})^2\tilde{L}_f^2\kappa_1^2\zeta_{k,t}^2\|D_t^R\|^2.\nonumber
			\end{align}
			{{\color{black}{Taking full expectation and noting that $(1-\alpha_t)^2 \le \min\{ 1-\alpha_t, 1 \}$, we obtain the desired result.}}}
		\end{proof}
		
		Equipped with Lemma \ref{lemma: bound estimator error}, we can establish the desired bound on $\Upsilon_{e}$.
		\begin{lemma}\label{lem: bound sum skt}
			Suppose that Assumptions {{\color{black}{\ref{assumption: level bound},}}} \ref{assumption1}, and \ref{assumption: stochastic} hold. Let $\{x_{k,t}\}$ be a sequence generated by Algorithm \ref{Algorithm: RADASTORM}. Then, we have
			\begin{align}\label{inequ: bound sum skt}
				\mathbb{E}\left[\Upsilon_{e}\right]\leq\sigma^2+\frac{18}{b}\sigma^2(TK)^{1/3}+\frac{18c_1}{b} (TK)^{2/9}\mathbb{E}\left[\Upsilon_{d}^{1/3}\right],
			\end{align}
			where 
			$
				c_1=\tilde{L}_f^2\kappa_1^2\xi_{0}^2
			$ and the expectation $\mathbb{E}[\cdot]$ is taken over all randomness in the course of Algorithm \ref{Algorithm: RADASTORM}.
		\end{lemma}
		\begin{proof}
			First, we deduce from \eqref{inequ: bound estimator error} that
			\begin{align}
				&\,\sum_{k=1}^{K}\sum_{t=1}^{T}\frac{\mathbb{E}[\|e_{k,t+1}\|^2]}{\alpha_{k,t}}\nonumber\\
				\leq&\,\sum_{k=1}^{K}\sum_{t=1}^{T}\left(\frac{1}{\alpha_{k,t}}-1\right)\mathbb{E}[\|e_{k,t}\|^2]+\frac{2\tilde{L}^2_f\kappa_1^2}{b}{{\color{black}{\mathbb{E}\!\left[\sum_{k=1}^{K}\sum_{t=1}^{T}\frac{\zeta_{k,t}^2}{\alpha_{k,t}}\|D_{k,t}^R\|^2\right]}}}+\frac{2}{b}\sigma^2\sum_{k=1}^{K}\sum_{t=1}^{T}\alpha_{k,t}\nonumber.
			\end{align}
			Rearranging terms results in 
			\begin{equation}
			\begin{aligned}\label{inequ: sum skt 1}
				&\,\mathbb{E}[\Upsilon_{e}]\\
				{{\color{black}{\leq}}}&\,\sum_{k=1}^{K-1}\sum_{t=1}^{T}\left(\frac{1}{\alpha_{k,t+1}}-\frac{1}{\alpha_{k,t}}\right)\mathbb{E}[\|e_{k,t+1}\|^2]+\sum_{t=1}^{T-1}\left(\frac{1}{\alpha_{K,t+1}}-\frac{1}{\alpha_{K,t}}\right)\mathbb{E}[\|{{\color{black}{e_{K,t+1}}}}\|^2]\\
				&\,+\frac{\mathbb{E}[\|e_{1,1}\|^2]}{\alpha_{1,1}}-\frac{\mathbb{E}[\|{{\color{black}{e_{K,T+1}}}}\|^2]}{\alpha_{K,T}}+\frac{2\tilde{L}_f^2\kappa_1^2}{b}{{\color{black}{\mathbb{E}\!\left[\sum_{k=1}^{K}\sum_{t=1}^{T}\frac{\zeta_{k,t}^2}{\alpha_{k,t}}\|D_{k,t}^R\|^2\right]}}}+\frac{2}{b}\sigma^2\sum_{k=1}^{K}\sum_{t=1}^{T}\alpha_{k,t}.
			\end{aligned}
			\end{equation}
		    Since the function $s\mapsto s^{2/3}$ is concave on $[0,+\infty)$, for $k = 1,2,\ldots,K$ and $t = 1,2,\ldots,T$, we have 
			\begin{align}
				\frac{1}{\alpha_{k,t+1}}-\frac{1}{\alpha_{k,t}}=&\,((k-1)T+t+1)^{2/3}-((k-1)T+t)^{2/3}\nonumber\\
				\leq&\,\frac{2}{3}\frac{1}{((k-1)T+t)^{1/3}}\leq\frac{2}{3}.\nonumber
			\end{align}
			Using the fact that $\alpha_{k,T+1}=\alpha_{k+1,1}$ and $e_{k, T + 1} = e_{k +1,1}$ for $k=1,2,\ldots,K-1$ and the definition of $\Upsilon_{e}$ in \eqref{equ: Upsilon gds}, we have
			\begin{align}
				&\,\sum_{k=1}^{K-1}\sum_{t=1}^{T}\left(\frac{1}{\alpha_{k,t+1}}-\frac{1}{\alpha_{k,t}}\right)\mathbb{E}[\|e_{k,t+1}\|^2]+\sum_{t=1}^{T-1}\left(\frac{1}{\alpha_{K,t+1}}-\frac{1}{\alpha_{K,t}}\right)\mathbb{E}[\|{{\color{black}{e_{K,t+1}}}}\|^2]\nonumber\\
				\leq&\,\frac{2}{3}\sum_{k=1}^{K-1}\sum_{t=1}^{T}\mathbb{E}[\|e_{k,t+1}\|^2]+\frac{2}{3}\sum_{t=1}^{T-1}\mathbb{E}[\|{{\color{black}{e_{K,t+1}}}}\|^2]
				\leq\,\frac{2}{3}\mathbb{E}[\Upsilon_{e}]-\frac23\mathbb{E}[\|e_{1,1}\|^2].\nonumber
			\end{align}
			Substituting the above into \eqref{inequ: sum skt 1} and noting that $\alpha_{1,1} = 1$, we obtain 
			\begin{equation*}
			\begin{aligned}
				\mathbb{E}[\Upsilon_{e}]\leq&\,\frac{2}{3}\mathbb{E}[\Upsilon_{e}]+\frac13\mathbb{E}[\|e_{1,1}\|^2]-\frac{\mathbb{E}[\|{{\color{black}{e_{K,T+1}}}}\|^2]}{\alpha_{K,T}}\\
				&\,+\frac{2\tilde{L}_f^2\kappa_1^2}{b}{{\color{black}{\mathbb{E}\!\left[\sum_{k=1}^{K}\sum_{t=1}^{T}\frac{\zeta_{k,t}^2}{\alpha_{k,t}}\|D_{k,t}^R\|^2\right]}}}+\frac{2}{b}\sigma^2\sum_{k=1}^{K}\sum_{t=1}^{T}\alpha_{k,t}.
			\end{aligned}
		\end{equation*}
			It follows that
			\begin{equation}
				\begin{aligned}\label{inequ: sum skt 2}
			\mathbb{E}[\Upsilon_{e}]\leq&\,\mathbb{E}[\|e_{1,1}\|^2]+\frac{6\tilde{L}_f^2\kappa_1^2}{b}\mathbb{E}\!\left[\sum_{k=1}^{K}\sum_{t=1}^{T}\frac{\zeta_{k,t}^2}{\alpha_{k,t}}\|D_{k,t}^R\|^2\right]+\frac{6}{b}\sigma^2\sum_{k=1}^{K}\sum_{t=1}^{T}\alpha_{k,t}.
		\end{aligned}
	\end{equation}
	
		Next, we bound the second term on the right-hand side of the above inequality.
		Let $k \in \{1,2,\ldots,K\}$ and $t \in \{1,2,\ldots,T\}$ be fixed. From \eqref{equ: zetakt}, we get  
		\begin{align}
				\frac{\zeta_{k,t}^2}{\alpha_{k,t}}=&\, \alpha_{k,t}^{-1/3} \xi_k^2 {{\color{black}{\left(\delta_0+\sum_{i=1}^{k-1}\sum_{j=1}^{T}\|D_{i,j}^R\|^2+\sum_{j=1}^{t}\|D_{k,j}^R\|^2\right)^{-2/3}}}}\nonumber\\
				\leq&\, \xi_{0}^2 (TK)^{2/9}  {{\color{black}{\left(\delta_0+\sum_{i=1}^{k-1}\sum_{j=1}^{T}\|D_{i,j}^R\|^2+\sum_{j=1}^{t}\|D_{k,j}^R\|^2\right)^{-2/3}}}},\nonumber
			\end{align}
			where the inequality is due to $\xi_k\leq\xi_{0}$ and $\alpha_{k,t}\geq (TK)^{-2/3}$.
			 Therefore,
			\begin{equation}\label{inequ: sum D}
				\begin{aligned}
					\sum_{k=1}^{K}\sum_{t=1}^{T}\frac{\zeta_{k,t}^2}{\alpha_{k,t}}\|D_{k,t}^R\|^2\leq{}&\sum_{k=1}^{K}\sum_{t=1}^{T}\frac{(TK)^{2/9} \xi_{0}^2\|D_{k,t}^R\|^2}{{{\color{black}{(\delta_0+\sum_{i=1}^{k-1}\sum_{j=1}^{T}\|D_{i,j}^R\|^2+\sum_{j=1}^{t}\|D_{k,j}^R\|^2)^{2/3}}}}}\\
					\overset{\text{(a)}}{\leq}{}&{3\xi_{0}^2(TK)^{2/9}}\Upsilon_{d}^{1/3},
				\end{aligned}
			\end{equation}
			where (a) follows from \eqref{inequ: shifted storm+} with $p=2/3$ and the definition of $\Upsilon_{d}$ in {{\color{black}{\eqref{equ: Upsilon gds}}}}.
			Additionally, we have
			\begin{align}\label{inequ: sum alphak}
				\sum_{k=1}^{K}\sum_{t=1}^{T}\alpha_{k,t} = \sum_{k=1}^{K}\sum_{t=1}^{T}\frac{1}{((k-1)T+t)^{2/3}}\leq{3}(TK)^{1/3}.
			\end{align}
			Substituting \eqref{inequ: sum D}, \eqref{inequ: sum alphak} into \eqref{inequ: sum skt 2} and noting that $\mathbb{E}[\|e_{1,1}\|^2]\leq\sigma^2$, we obtain the desired inequality \eqref{inequ: bound sum skt}.
		\end{proof}
		
		\subsection{Complexity Bounds}
		Our goal now is to derive the iteration complexity and SFO complexity of Algorithm \ref{Algorithm: RADASTORM} for computing an approximate Riemannian-stochastic-stationary point of problem \eqref{prob:p1}, which is defined as follows.
			
		\begin{definition}\label{stationarypoint}
			Given $\epsilon>0$, the point $x\in\mathcal{M}$ is an $\epsilon$-Riemannian-stochastic-stationary ($\epsilon$-RSS) point of problem \eqref{prob:p1} if there exists a point $p\in\mathcal{E}_2$ such that
			\begin{equation*}\label{Def optstationary}
				\max\!\left\{\mathbb{E}\left[\mathrm{dist}\left(0,\,\rgrad f(x) + \proj_{\mathrm{T}_{x}\mathcal{M}}\left(\nabla \mathcal{A}(x)^\top \partial h(p)\right)\right)\right],\,\mathbb{E}[\left\|p-\mathcal{A}(x)\right\|]\right\}\leq\epsilon,
			\end{equation*}
			where the expectation $\mathbb{E}[\cdot]$ is taken over all randomness in the course of generating $x$.
		\end{definition}
		We remark that the above stationarity notion is generalized from those introduced in \cite{deng2024oracle,beck2023dynamic,li2023riemannian} for deterministic settings.
		
		With the above preparations, we are ready to prove the main result of this paper.
		\begin{theorem}\label{theo: complexity}
			 Suppose that Assumptions \ref{assumption: level bound}, \ref{assumption1}, and \ref{assumption: stochastic} hold. Given $\epsilon > 0$, let $\{x_k\}$ be the sequence generated by Algorithm \ref{Algorithm: RADASTORM}.
			Then, Algorithm \ref{Algorithm: RADASTORM} returns an $\epsilon$-RSS point of problem \eqref{prob:p1} in at most ${\mathcal{O}}(\epsilon^{-3})$ iterations while making at most ${\mathcal{O}}(\epsilon^{-3})$ SFO calls.
		\end{theorem}
		\begin{proof}
			Let $\hat{k}$ be uniformly sampled from $\{1,2,\ldots,K\}$. We first prove that
			\begin{equation*}%
				\mathbb{E}[\|\rgrad\Phi_{\hat{k}}(x_{\hat{k}})\|^2]=\mathcal{O}(K^{-2/3}).
			\end{equation*}
		   For each $k \in \{1,2,\ldots,K\}$ and $t \in \{1,2,\ldots,T\}$, upon recalling the definitions of $D^R_{k,t}$ in \eqref{equ:Dkt} and $e_{k,t}$ in \eqref{equ: Upsilon gds}, and  using  $\rgrad\Phi_{k}(x_{k,t}) =\proj_{\mathrm{T}_{x_{k,t}}\mathcal{M}}(\nabla \Phi_k(x_{k,t}))$ together with \eqref{nabla Phik},   we obtain
			\begin{align}%
				\|D^R_{k,t}\|^2 = &\, \|\proj_{\mathrm{T}_{x_{k,t}}\mathcal{M}}(e_{k,t} )+ \rgrad\Phi_{k}(x_{k,t})\|^2\nonumber\\
				\leq&\,\big\|e_{k,t} + \rgrad\Phi_{k}(x_{k,t})\big\|^2 \nonumber\\
				\leq&\,2\|e_{k,t}\|^2 + 2\|\rgrad\Phi_{k}(x_{k,t})\|^2,\nonumber %\label{equ:D-upper-bound}
			\end{align}
			where the first inequality follows from the nonexpansiveness of the projection operator. It follows from the definitions of $\Upsilon_d$, $\Upsilon_{g}$, and $\Upsilon_{e}$ in \eqref{equ: Upsilon gds} that
			\begin{equation}
				\label{equ:Upsilon_d:bound1}
				\mathbb{E}[\Upsilon_d] \leq 2(\mathbb{E}[\Upsilon_e]+ \mathbb{E}[\Upsilon_g]).
			\end{equation}
			Taking full expectation in  \eqref{inequ: pre decent of Phi} and using  Jensen's inequality, we obtain
			\begin{align}\label{inequ: EUpsilong}
				\mathbb{E}[\Upsilon_{g}]
				\leq \frac{2(\Upsilon+\Delta_\Phi)}{\xi_{\min}}(TK)^{2/9}\, (\delta_0+\mathbb{E}[\Upsilon_{d}])^{1/3}
				+ \frac{3\xi_{0}L_c}{2\lambda_{1}}K^{1/9}\,(\mathbb{E}[\Upsilon_{d}])^{2/3}
				+ \mathbb{E}[\Upsilon_{e}].
			\end{align}
			Substituting \eqref{inequ: bound sum skt}, \eqref{inequ: EUpsilong} into \eqref{equ:Upsilon_d:bound1} and using $(\delta_0+z)^{1/3}\leq\delta_0^{1/3}+z^{1/3}$, $(TK)^{2/9}\leq(TK)^{1/3}$, we get 
			\begin{align}\label{inequ: Upsilond 1}
				\mathbb{E}[\Upsilon_{d}]
				\leq c_2 (TK)^{2/9}\,(\mathbb{E}[\Upsilon_{d}])^{1/3}
				+c_3 K^{1/9}\,(\mathbb{E}[\Upsilon_{d}])^{2/3}
				+ c_4  (TK)^{1/3} + 4 \sigma^2,
			\end{align}
			where 
			\[
			c_2 = \frac{4(\Upsilon+\Delta_\Phi)}{\xi_{\min}} + \frac{72}{b}c_1, \quad c_3 =  \frac{3\xi_{0}L_c}{\lambda_1},\quad c_4 =  \frac{72}{b}{\sigma^2}{{\color{black}{+\frac{4(\Upsilon+\Delta_\Phi)}{\xi_{\min}}\delta_0^{1/3}}}}.  
			\]
			Note that given $c > 0$ and $\gamma \in (0,1)$, the function $z \mapsto p(z) := cz^{\gamma}$ is concave. This implies that
			\[
			p(z)\le p(z_0)+p'(z_0)(z-z_0), \quad \forall\,z>0.
			\]
			Choosing $z_0 = (4c\gamma)^{1/(1-\gamma)}$, we have $p'(z_0)=\tfrac14$, which gives
			\begin{equation}\label{equ:pz}
				c z^\gamma \le (1-\gamma)(4\gamma)^{\frac{\gamma}{1-\gamma}} c^{\frac{1}{1 - \gamma}}+\tfrac14 z, \quad \forall\,z>0.
			\end{equation}
			Applying \eqref{equ:pz} with $c  = c_2 (TK)^{2/9}$,  $z = \mathbb{E}[\Upsilon_{d}]$, and $\gamma = 1/3$ yields  
			\[
			c_2 (TK)^{2/9} (\mathbb{E}[\Upsilon_{d}])^{1/3} \leq \frac{4}{3\sqrt{3}} c_2^{3/2} (TK)^{1/3} + \frac14 \mathbb{E}[\Upsilon_{d}].
			\]
			Applying \eqref{equ:pz} again with $c  = c_3 (K)^{1/9}$,  $z = \mathbb{E}[\Upsilon_{d}]$, and $\gamma = 2/3$ yields 
			\[
			c_3 K^{1/9}\,(\mathbb{E}[\Upsilon_{d}])^{2/3} \leq \frac{64}{27}\,c_3^{3} K^{1/3} +  \frac14 \mathbb{E}[\Upsilon_{d}]. 
			\]
			Substituting these two bounds into \eqref{inequ: Upsilond 1} and rearranging give 
			\[
			\mathbb{E}[\Upsilon_{d}]  \leq \left(   \frac{8}{3\sqrt{3}} c_2^{3/2} T^{1/3} + 2 c_4 T^{1/3} + {{\color{black}{\frac{128}{27}c_3^3}}}  \right) K^{1/3} + 8 \sigma^2,
			\]
			which implies that $\mathbb{E}[\Upsilon_d]= \mathcal{O}(K^{1/3}).$
			Together with  \eqref{inequ: bound sum skt}, this yields   $\mathbb{E}[\Upsilon_e] = \mathcal{O}(K^{1/3})$. Thus, we obtain from \eqref{inequ: EUpsilong} that  
			\begin{align}
				\mathbb{E}[\Upsilon_{g}] = \mathcal{O}(K^{1/3}).\nonumber
			\end{align}
			Upon recalling the definition of $\Upsilon_{g}$ in \eqref{equ: Upsilon gds} and the fact that $\hat{k}$ is uniformly sampled from $\{1,2,\ldots,K\}$, we have		
			\begin{equation}\label{inequ: decrease rate of EgradPhi}
				\mathbb{E}[\|\rgrad\Phi_{\hat{k}}(x_{\hat{k}})\|^2]\leq \mathbb{E}[\Upsilon_{g}]/K= \mathcal{O}(K^{-2/3}).
			\end{equation}

			Next, let %
		$
				p_k:=\prox_{(\lambda_k+\beta_{k})h}(\mathcal{A}(x_k)+\beta_{k}y_k)
			$ for $k=1,2,\ldots,K$. We show that $\mathbb{E}[ \| p_{\hat{k}} - \mathcal{A}(x_{\hat{k}}) \| ] = \mathcal{O}(K^{-1/3})$.
			From \eqref{nabla Phik} and Theorem \ref{gradient of Moreau}, we have
			\begin{align}
				\rgrad\Phi_{k}(x_k)&=\rgrad f(x_k) + \proj_{\mathrm{T}_{x_k}\mathcal{M}}\left(\nabla \mathcal{A}(x_k)^\top y_{k+\frac12}\right),\label{equ: gradPhi with y}\\
				y_{k+\frac12}&\in\partial h(p_k),\label{equ: yk+12 in ph}
			\end{align}
			where $y_{k+\frac12}$ is defined in \eqref{yk+1/2}.
			Therefore, it follows from \eqref{inequ: decrease rate of EgradPhi}, \eqref{equ: gradPhi with y}, \eqref{equ: yk+12 in ph}, and the inequality $\mathbb{E}[z]^2\leq\mathbb{E}[z^2]$ that 
			\begin{equation}
			\begin{aligned}\label{inequ: bound dist grad}
				&\,\mathbb{E}\left[\mathrm{dist}\left(0,\,\rgrad f(x_{\hat{k}}) + \proj_{\mathrm{T}_{x_{\hat{k}}}\mathcal{M}}\left(\nabla \mathcal{A}(x_{\hat{k}})^\top \partial h(p_{\hat{k}})\right)\right)\right]^2\\
				\leq&\,\mathbb{E}[\|\rgrad\Phi_{\hat{k}}(x_{\hat{k}})\|^2]=\mathcal{O}(K^{-2/3}).%
			\end{aligned}
			\end{equation}
			Moreover, for all $k\geq1$, we have
			\begin{align}
				\|p_k-\mathcal{A}(x_k)\|=&\,\left\|\prox_{(\lambda_k+\beta_{k})h}(\mathcal{A}(x_k)+\beta_{k}y_k)-\mathcal{A}(x_k)\right\|\nonumber\\
				\overset{(\text{a})}{=}&\,\left\|\beta_{k}y_k-(\lambda_k+\beta_{k})\prox_{h^*/(\lambda_k+\beta_k)}\left(\frac{\mathcal{A}(x_k)+\beta_{k}y_k}{\lambda_k+\beta_{k}}\right)\right\|\nonumber\\
				\overset{(\text{b})}{\leq}&\,(\lambda_k+2\beta_{k})L_h,\nonumber
			\end{align}
			where (a) is from \eqref{Moreau Decomposition} and (b) is due to the $L_h$-Lipschitz continuity of $h$.
			Therefore, since $\hat{k}$ is uniformly sampled from $\{1,2,\ldots,K\}$, we have
			\begin{align}\label{inequ: bound p-Ax}
				\mathbb{E}[\|p_{\hat{k}}-\mathcal{A}(x_{\hat{k}})\|]\leq&\,\frac{L_h}{K}\sum_{k=1}^{K}(\lambda_k+2\beta_{k})\nonumber\\
				=&\,\frac{L_h}{K}\sum_{k=1}^{K}\left(\frac{\lambda_1}{k^{1/3}}+2\beta_{k}\right)\nonumber\\
				\overset{\text{(a)}}{\leq}&\,\frac{\Upsilon}{K}+\frac{3L_h\lambda_1(K+1)^{2/3}}{2K}=\mathcal{O}(K^{-1/3}),%
			\end{align}
			where (a) is due to the definition of $\Upsilon$ in \eqref{equ: MPhi and Upsilon}.
			Combining \eqref{inequ: bound dist grad} and \eqref{inequ: bound p-Ax}, we conclude that $x_{\hat{k}}$ is an $\epsilon$-RSS point of problem \eqref{prob:p1} with $\hat{k}\leq K=\mathcal{O}(\epsilon^{-3})$. Since Algorithm \ref{Algorithm: RADASTORM} only requires $\mathcal{O}(1)$ SFO calls at each iteration, the total number of SFO calls is also $\mathcal{O}(\epsilon^{-3})$.
		\end{proof}
		
		Theorem \ref{theo: complexity} shows that both the iteration complexity and  SFO complexity of the proposed StoRADA-RM are $\mathcal{O}(\epsilon^{-3})$. It is worth noting that the iteration complexity coincides with the best known in the literature for deterministic Riemannian nonsmooth composite optimization (see, e.g., \cite{beck2023dynamic,deng2024oracle,xu2025oracle}), while the SFO complexity not only matches the best known for Riemannian nonsmooth composite expectation optimization (see, e.g., \cite{deng2025single,wang2022riemannian}) but also attains the optimal lower bound established in \cite{arjevani2023lower} for smooth nonconvex  optimization with stochastic first-order algorithms  under the mean-squared smoothness assumption \eqref{inequ: tilde Lf}.
		
			It is also worth emphasizing that the strategy for establishing the iteration complexity of the proposed StoRADA-RM is significantly different from that for its deterministic counterpart RADA-RGD \cite{xu2024riemannian}.
			In particular, the sufficient decrease property that is central to the complexity analysis of RADA no longer holds due to gradient estimation errors, which accumulate across the inner iterations of StoRADA-RM. To address this challenge, we  characterize the impact of the accumulated gradient estimation errors on the progress made by StoRADA-RM towards stationarity. 
		Furthermore, our analysis allows for considerable flexibility in the choice of stepsizes $\{ \zeta_{k,t} \}$. Unlike RADA-RGD, where $\{ \zeta_{k,t} \}$ must satisfy a line-search condition requiring extra objective function evaluations, StoRADA-RM avoids this cost while preserving the same iteration  complexity guarantee.

		\section{Numerical Results}\label{sec: numerical results}
		In this section, we report numerical results of our proposed StoRADA-RM for solving sparse PCA and CISE problems. The corresponding experiments were implemented in MATLAB 2023b and evaluated on Apple M2 Pro CPU.
		
		We begin by specifying the default parameter settings used in our experiments. 
		For StoRADA-RM, we adopt the update scheme for $\{ \beta_k \}$ from RADA-RGD \cite{xu2024riemannian} to improve practical performance.
		Given constants $\tau_1\in (0,1)$ and $\tau_2 \in (0,1)$, we set
		\begin{equation}\label{RADA up beta}
			\beta_{k+1} = \frac{\beta_1^{(k+1)}}{(k+1)^\rho}\quad \mbox{with}\quad \beta_1^{(k+1)}=\left\{
			\begin{aligned}
				\tau_2\beta_{1}^{(k)}, &\quad \text {if } \delta_{k+1}\geq\tau_1\delta_k; \\ 
				\beta_{1}^{(k)},&\quad \text {if } \delta_{k+1}<\tau_1\delta_k,
			\end{aligned}
			\right.
		\end{equation}
		for $k \ge 1$,  where $\beta_1^{(1)} = \beta_1$ and  $\delta_{k+1}$ is defined as
		\begin{equation}\label{delta}
			{\delta_{k+1}} := \|(\lambda_k+\beta_k)y_{k+1}-\beta_{k}y_{k}\|_\infty, \quad\forall\,k\geq0.
		\end{equation} 
		Here, we adopt the convention that $\lambda_0 = \lambda_1$, $\beta_0 = \beta_1$, and $y_0 = y_1$ when computing $\delta_1$. For our tests, we set $\rho=1.5$, $\tau_1=0.999$, and $\tau_2=0.9$. 
		We initialize  $\lambda_{1}=10^{-6}$ and use $\lambda_k = \lambda_1k^{-1/3}$ for $k \ge 2$. We choose $\xi_k=\max\{\xi_{0}/k^{1/2},0.01\}$ for $k\ge1$, where $\xi_{0}$ is tuned for each application problem. We  adopt the same smoothing parameters $\{ \lambda_k \}$ and stepsizes $\{ \zeta_k \}$ for RSSM \cite{deng2025single} to improve its practical performance.
		For StoManIAL, we set the penalty parameter of the AL function as $\sigma_k=1000\cdot2^{2k/7}$, the dual stepsize as 
		\begin{equation*}
			\beta_{k}=1000\min\left\{\frac{\|\mathcal{A}(x_1 )- y_1\|\log^22}{\|\mathcal{A}(x_{k+1})-y_{k+1}\|(k+1)^2\log(k+2)},1\right\},
		\end{equation*}
		and the maximum number of inner iterations as $T_k = \min\{2^k, 5000\}$, where $k\ge1$. Compared with the original settings $\sigma_k = 2^{2k/7}$ and $T_k = 2^k$ of StoManIAL \cite{deng2024oracle}, introducing the coefficient 1000 enables StoManIAL to generate sparser solutions in fewer iterations, while the cap 5000 prevents the algorithm from taking too much time in one outer iteration, especially when $k$ is large. 
		Finally, we terminate StoRADA-RM, StoManIAL, and R-ProxSPB when $\|x_{k+1}-x_k\| \leq \text{tol}_1$ with $\text{tol}_1>0$. Note that these three algorithms all perform multiple Riemannian stochastic gradient steps to generate $x_{k+1}$, while RSSM only performs a single such step. Thus, the gap $\|x_{k+1}-x_k\|$ generated by the former can be relatively larger than the latter. To ensure fairness in comparison, we terminate RSSM when $\|x_{k+1}-x_k\| \leq \text{tol}_2$ with $\text{tol}_2 < \text{tol}_1$.
		\subsection{Sparse PCA}\label{subsection: test spca}
		In this subsection, we compare the proposed StoRADA-RM with RSSM \cite{deng2025single}, R-ProxSPB \cite{wang2022riemannian}, and StoManIAL \cite{deng2024oracle} on the sparse PCA problem. 
		
		Let $a\in\mathbb{R}^d$ represent a random data point following a certain distribution $\mathcal{D}$.
		To achieve a good balance between dimension reduction and interpretability, sparse PCA seeks principal components with few nonzero components. 
		This motivates the following formulation of sparse PCA \cite{jolliffe2003modified}:
		\begin{equation}\label{prob: SPCA}
			\min_{X\in\mathcal{S}(d,r)} \left\{\mathbb{E}_{a\sim\mathcal{D}}[\|a-XX^\top a\|^2]+\mu\|X\|_1\right\}.
		\end{equation}		
		Here, %
		$ \mathcal{S}(d,r) = \{X \in \mathbb{R}^{d \times r}\mid X^\top X = I_r\} $ is the Stiefel manifold with $I_r$ being the $r$-by-$r$ identity matrix, $\mu>0$ is the weighting parameter, and $\|X\|_1=\sum_{i,j}|X_{ij}|$ is the $\ell_1$-norm of the matrix $X$. When the support $\Omega = \{a_1, a_2, \ldots, a_N\}$ is finite and $a$ is uniformly distributed on $\Omega$, the expectation term $\mathbb{E}_{a\sim\mathcal{D}}[\|a-XX^\top a\|^2]$ reduces to the finite-sum  $\frac{1}{N}\sum_{i=1}^{N}\|a_i-XX^\top a_i\|^2$.

		We first consider the finite-sum case of problem \eqref{prob: SPCA} and perform tests on synthetic datasets and the real dataset coil-100 \cite{nene1996columbia}, which contains 7200 image samples of 100
		objects taken from different angles with $d = 1024$. For the synthetic data, we generate the data matrix $A=[a_1,a_2,\ldots,a_N]$ with $N=10000$ and $d\in\{400,600,800,1000\}$ as in \cite{zhou2023semismooth}. %
			For StoRADA-RM, we set $\beta_{1}=1$, $T=10$, $\xi_0=1$, {{\color{black}{and $\delta_0=1$}}}. For the inner subproblem solver of StoManIAL, we fine-tuned the parameters and found the following settings worked well in our tests:
		\begin{equation}\label{parameter zeta for RALM}
			\zeta_{k,t}={\zeta_{k}}{\left(1+\sum_{i=1}^t	\sum_{a\in S_{k,i}} \|\rgrad\mathcal{L}_k(x_{k,i};a)\|/b\right)^{-{1}/{3}}},\quad\zeta_{k}={0.05}\cdot{2^{-{k}/{3}}},
		\end{equation}
		and $\alpha_{k,t}=100\zeta_{k,t}^2/k^{1/2}$, where $k \ge 1$ and $t=1,2,\ldots,T$. 
		For R-ProxSPB, we set $\gamma=1/{L}_f$ and $ \eta=0.1 $ after  fine-tuning to enhance its practical performance. Note that these values violate the theoretical requirements of R-ProxSPB, where $\gamma$ is required to be 0.4 and $\eta$ depends on problem-specific parameters, $\kappa_2$, $L_h$, and a uniform Hessian bound for $f$.  We set the batch size $b=100$ for StoRADA-RM, StoManIAL, and RSSM. For R-ProxSPB, we set the large-batch size to $N$ in every 100 iterations and use a small-batch size of 100 in the other iterations.
		For the stopping criteria, we set $\text{tol}_1=10^{-3}$ and $\text{tol}_2=0.2\text{tol}_1$.
		\begin{table}
			\centering
			\fontsize{7pt}{\baselineskip}\selectfont
			\caption{Average performance comparison on sparse PCA in the finite-sum case.}
			\tabcolsep=0.025cm
			\resizebox{\textwidth}{!}{
				\begin{tabular}{cccccc|ccccc|ccccc|ccccc}
					\hline
					&\multicolumn{5}{c}{StoRADA-RM}&\multicolumn{5}{c}{RSSM}  & \multicolumn{5}{c}{StoManIAL} & \multicolumn{5}{c}{R-ProxSPB}  \\
					\cline{2-21}
					& $\Phi$  & spar& loss & cpu & iter  & $\Phi$  & spar& loss & cpu & iter  & $\Phi$  & spar& loss & cpu & iter  & $\Phi$  & spar& loss & cpu & iter \\
					\hline
					$\mu$ & \multicolumn{20}{c}{synthetic, $d=1000, r=15$}\\
					\hline
					0.4 & 291.7 & 41.3 & 179.8 & 9 & 621 &  296.9 & 40.4 & 185.7 & 19 & 8862 & 298.3 &  29.5 & 180.1 & 225 & 43 & 299.1 & 41.0 & 189.0 & 26 & 2318 \\ 
					0.6 & 329.3 & 54.3 & 186.0 & 8 & 554 &  338.5 & 54.2 & 195.5 & 24 & 11370 & 335.6 &  50.2 & 184.7 & 231 & 44 & 335.9 & 55.3 & 195.8 & 29 & 2575 \\ 
					0.8 & 384.0 & 67.4 & 229.2 & 8 & 550 &  396.5 & 67.4 & 242.5 & 30 & 14790 & 387.2 &  63.4 & 227.5 & 223 & 44 & 386.3 & 68.9 & 245.1 & 36 & 3223 \\ 
					\hline
					$d$ & \multicolumn{20}{c}{synthetic, $\mu=0.5, r=20$}\\
					\hline
					400 &  131.3 & 71.9 & 60.6 & 3 & 456 &  144.2 & 68.1 & 64.6 & 10 & 9782 & 134.1 &  70.1 & 60.5 & 90 & 44 & 134.3 & 73.3 & 64.3 & 30 & 3578 \\ 
					600 & 184.5 & 67.5 & 88.4 & 5 & 516 &  198.4 & 65.2 & 95.5 & 16 & 10534 & 187.6 &  65.0 & 87.8 & 150 & 44 & 188.4 & 68.5 & 94.2 & 41 & 3586 \\ 
					800 & 250.6 & 65.3 & 132.0 & 7 & 547 &  264.2 & 63.1 & 141.3 & 22 & 11276 & 256.3 &  60.4 & 128.6 & 204 & 44 & 256.2 & 65.9 & 139.8 & 47 & 3348 \\ 
					\hline
					$\mu$ & \multicolumn{20}{c}{coil-100, $d=1024, r=15$}\\
					\hline
					0.4 &  245.0 & 63.1 & 172.4 & 8 & 525 &  258.5 & 57.6 & 178.2 & 20 & 9663 & 256.6 &  54.0 & 172.6 & 213 & 41 & 252.4 & 61.2 & 176.9 & 45 & 3641 \\ 
					0.6 & 277.9 & 74.4 & 183.5 & 6 & 429 &  297.7 & 69.6 & 196.3 & 27 & 12874 & 286.8 &  64.2 & 184.3 & 216 & 42 & 287.6 & 73.0 & 192.9 & 43 & 3329 \\ 
					0.8 & 308.6 & 79.9 & 192.2 & 6 & 409 &  334.9 & 75.9 & 214.6 & 33 & 15818 & 319.0 &  60.6 & 197.4 & 218 & 42 & 321.4 & 79.2 & 210.4 & 43 & 3342 \\ 
					\hline
					$r$ & \multicolumn{20}{c}{coil-100, $d=1024, \mu=0.8$}\\
					\hline
					10 & 313.8 & 68.2 & 213.8 & 5 & 386 &  323.8 & 66.8 & 223.7 & 22 & 13031 & 320.9 &  47.4 & 215.7 & 168 & 41 & 323.1 & 67.4 & 225.6 & 17 & 2755 \\ 
					15 & 308.5 & 80.2 & 192.2 & 6 & 394 &  338.5 & 76.2 & 218.3 & 32 & 15796 & 316.4 &  58.6 & 194.9 & 210 & 42 & 322.0 & 79.2 & 212.5 & 44 & 3299 \\ 
					20 & 307.4 & 85.8 & 177.4 & 6 & 394 &  360.4 & 80.3 & 217.9 & 41 & 17966 & 317.1 &  84.0 & 186.1 & 247 & 43 & 320.6 & 84.6 & 198.3 & 82 & 3925 \\ 
					\hline
				\end{tabular}
			}
			\label{Tab:  SPCA  aver10}
		\end{table}

		The average results of 20 runs with different data matrices and initial points are presented in Table \ref{Tab:  SPCA  aver10}, where $\Phi$ denotes the objective value of problem \eqref{prob: SPCA}, ``cpu" represents the cpu time in seconds, and “iter" denotes the outer iteration number.
		We also compare the reconstruction loss value (denoted by “loss"), defined as $\frac{1}{N}\sum_{i=1}^{N}\|a_i-XX^\top a_i\|^2$, and the sparsity of $X$ (denoted by “spar''), measured as the percentage of entries with absolute value less than $10^{-5}$. 
		From Table \ref{Tab:  SPCA  aver10}, we observe that the proposed single-loop StoRADA-RM algorithm always returns the best solutions in terms of the objective value $\Phi$ and is the most efficient among the compared algorithms. In particular, although both StoRADA-RM and RSSM are single-loop algorithms with the same optimal SFO complexity, the former produces solutions with higher sparsity and lower reconstruction loss and does so in significantly less time. The superior performance stems from its ability to perform multiple Riemannian stochastic gradient steps when solving the subproblems, thereby yielding more accurate $x$-updates and substantially reducing the number of outer iterations. 
		Moreover, when compared with the two nested-loop algorithms StoManIAL and R-ProxSPB, the efficiency of StoRADA-RM becomes even more pronounced, particularly in high-dimensional settings with large $d$ and $r$.
		
				\begin{figure}[h]
			\centering
			\begin{subfigure}
				\centering
				\includegraphics[scale=0.3]{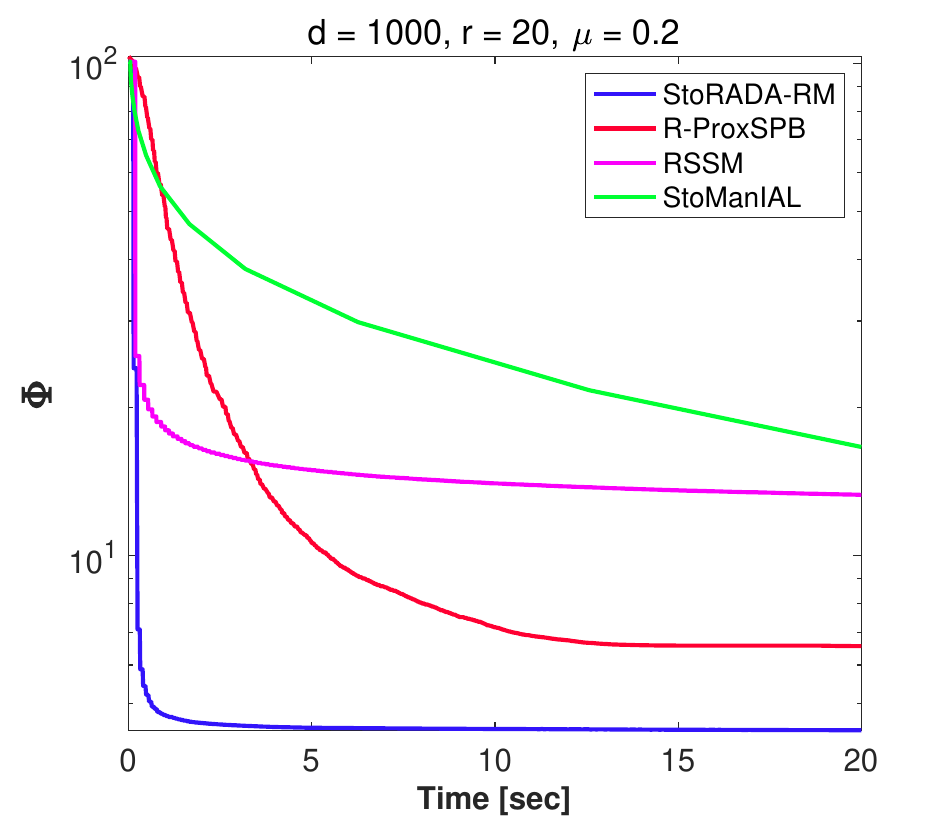}
				\end{subfigure}
				\begin{subfigure}
					\centering
					\includegraphics[scale=0.3]{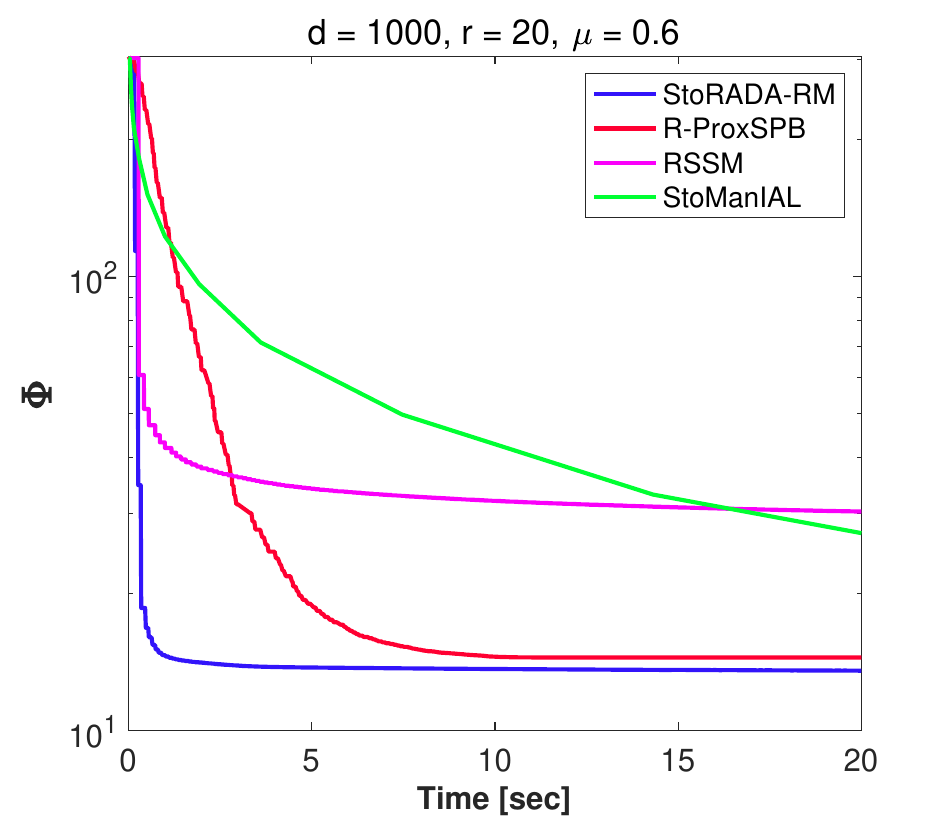}
				\end{subfigure}
				\caption{Objective value \(\Phi\) versus runtime comparison on sparse PCA in the general case.}
				\label{Spcar}
			\end{figure}

		Next, we consider the general case of problem \eqref{prob: SPCA}, where the data point $a$ is randomly generated at each update. Specifically, we first generate an auxiliary dataset \(\{a_i\}_{i=1}^{N_0}\) with \(N_0=200\), following the same procedure used for the finite-sum case, and set
		\[
		\bar a=\frac{1}{N_0}\sum_{i=1}^{N_0}a_i,
		\quad \text{and} \quad
		\Sigma=\frac{1}{N_0-1}\sum_{i=1}^{N_0}(a_i-\bar a)(a_i-\bar a)^\top+10^{-6}I_d.
		\]
		At each stochastic gradient evaluation, the required data samples are then drawn independently from \(\mathcal{N}(\mathbf{0},\Sigma)\).
		For R-ProxSPB, we set the large-batch size to $40000$ in every 200 iterations and use a small-batch size of 200 in the other iterations.
		All compared algorithms are executed for 20 seconds, and the other parameter settings remain the same as in the finite-sum case.
		We plot the objective value $\Phi$ against the runtime for two representative instances in Figure \ref{Spcar}.
		As shown in the figure, our proposed StoRADA-RM reduces the objective value significantly faster than the other compared algorithms and can find the best solution in terms of the objective value.

			\subsection{Results on CISE}
			
			In this subsection, we compare the proposed StoRADA-RM with RSSM \cite{deng2025single}, R-ProxSPB \cite{wang2022riemannian}, and StoManIAL \cite{deng2024oracle} on the CISE problem.
			The CISE problem, which aims at simultaneously achieving sparse sufficient dimension reduction and screening out irrelevant and redundant variables efficiently, admits the following formulation \cite{chen2010cordinate}:
			\begin{equation}\label{prob: l21PCA}
				\min_{X\in\mathcal{S}(d,r)} \left\{\mathbb{E}_{a\sim\mathcal{D}}[\|a-XX^\top a\|^2]+\hat{\mu}\|X\|_{2,1}\right\}.
			\end{equation}
			Here, $\|X\|_{2,1}:=\sum_{i=1}^{d}\sqrt{\sum_{j=1}^{r}X_{ij}^2}$ is the $\ell_{2,1}$-norm of the matrix $X$, which can shrink the corresponding row vectors of irrelevant variables to zero. We set
			$\hat{\mu} = \mu\sqrt{r+\ln(d)}$ as \cite{xiao2021exact}, where the constant $\mu>0$ controls the sparsity.
			
						\begin{table}
				\centering
				\fontsize{7pt}{\baselineskip}\selectfont
				\caption{Average performance comparison on CISE in the finite-sum case.}
				\tabcolsep=0.025cm
				\resizebox{\textwidth}{!}{
					\begin{tabular}{cccccc|ccccc|ccccc|ccccc}
						\hline
						&\multicolumn{5}{c}{StoRADA-RM}&\multicolumn{5}{c}{RSSM}  & \multicolumn{5}{c}{StoManIAL} & \multicolumn{5}{c}{R-ProxSPB}  \\
						\cline{2-21}
						& $\Phi$  & spar & loss & cpu & iter  & $\Phi$  & spar & loss & cpu & iter  & $\Phi$  & spar & loss & cpu & iter  & $\Phi$  & spar & loss & cpu & iter \\
						\hline
						$r$ & \multicolumn{20}{c}{synthetic, $d=1000, \mu=1.4$}\\
						\hline
						8 &  531.0 & 64.3 & 402.6 & 9 & 130 &  532.1 & 20.9 & 403.1 & 54 & 29329 & 537.6 &  48.8 & 390.4 & 127 & 35 & 531.9 & 64.7 & 403.1 & 53 & 992 \\ 
						10 & 545.6 & 76.9 & 446.3 & 14 & 192 &  548.3 & 25.1 & 447.1 & 66 & 34281 & 557.4 &  59.5 & 424.7 & 146 & 36 & 546.7 & 77.1 & 447.1 & 58 & 1156 \\ 
						12 & 562.0 & 82.0 & 461.4 & 12 & 169 &  565.6 & 25.8 & 461.3 & 75 & 38307 & 566.2 &  75.4 & 453.1 & 151 & 36 & 563.4 & 82.2 & 461.6 & 55 & 1015 \\ 
						\hline
						$d$ & \multicolumn{20}{c}{synthetic, $\mu=1.3, r=12$}\\
						\hline
						400 &  260.8 & 83.9 & 179.7 & 5 & 193 &  264.0 & 12.3 & 179.9 & 21 & 21506 & 264.4 &  78.6 & 175.1 & 57 & 38 & 261.7 & 84.2 & 180.4 & 9 & 624 \\ 
						600 &  362.2 & 93.3 & 287.8 & 7 & 175 &  365.7 & 20.7 & 287.9 & 42 & 31139 & 363.6 &  90.8 & 285.2 & 89 & 37 & 363.4 & 93.3 & 288.0 & 13 & 789 \\ 
						800 & 460.1 & 91.5 & 381.3 & 11 & 184 &  463.5 & 26.8 & 381.1 & 63 & 37234 & 461.6 &  89.6 & 375.9 & 129 & 38 & 460.8 & 91.6 & 381.5 & 28 & 865 \\ 
						\hline
						$r$ & \multicolumn{20}{c}{coil-100, $d=1024, \mu=1.3$}\\
						\hline
						8 & 562.5 & 31.6 & 367.4 & 14 & 210 &  564.4 & 1.1 & 368.1 & 25 & 13541 & 577.4 &  9.4 & 332.6 & 114 & 32 & 563.7 & 31.6 & 367.6 & 66 & 814 \\ 
						10 & 573.4 & 64.7 & 442.9 & 22 & 307 &  583.0 & 4.2 & 409.6 & 45 & 22787 & 596.5 &  23.1 & 348.9 & 146 & 35 & 579.1 & 60.4 & 422.2 & 87 & 896 \\ 
						12 & 575.7 & 93.0 & 505.4 & 33 & 462 &  593.3 & 13.0 & 459.4 & 67 & 33803 & 615.2 &  33.2 & 370.5 & 154 & 37 & 584.8 & 85.3 & 475.5 & 88 & 1078 \\ 
						\hline
						$\mu$ & \multicolumn{20}{c}{coil-100, $d=1024, r=9$}\\
						\hline
						1.2 &  557.5 & 25.8 & 358.5 & 16 & 228 &  560.4 & 0.9 & 359.8 & 22 & 11476 & 574.6 &  10.9 & 313.7 & 145 & 36 & 558.9 & 26.3 & 360.2 & 98 & 873 \\ 
						1.3 & 573.5 & 37.9 & 375.8 & 15 & 202 &  574.8 & 1.3 & 375.0 & 30 & 15228 & 586.3 &  16.7 & 344.5 & 147 & 35 & 574.3 & 39.8 & 381.1 & 78 & 813 \\ 
						1.4 & 572.7 & 79.5 & 462.5 & 18 & 270 &  578.6 & 10.8 & 443.6 & 57 & 30234 & 599.2 &  23.5 & 361.8 & 139 & 35 & 573.0 & 79.1 & 463.1 & 47 & 934 \\ 
						\hline
					\end{tabular}
				}
				\label{Tab:  21PCA}
			\end{table}
			
			We consider both the finite-sum and general cases of problem \eqref{prob: l21PCA} as in Section \ref{subsection: test spca} and conduct experiments following the setup therein. 
			For StoRADA-RM, we set $T=50$ and $\xi_0=0.5$. For the inner subproblem solver of StoManIAL, namely RSTORM, we  fine-tuned the parameters and adopted the scheme \eqref{parameter zeta for RALM} with $\zeta_k = 0.5 \cdot 2^{-k/3}$. For R-ProxSPB, we set  $\eta=0.2$ after careful tuning. For the finite-sum case, the stopping criteria are set as $\text{tol}_1=10^{-3}$ and $\text{tol}_2=0.4\text{tol}_1$.
			The other parameter settings are identical to those in Section \ref{subsection: test spca}. The average results of 20 runs with different data matrices and initial points for the finite-sum case are presented in Table \ref{Tab:  21PCA}, where “spar" is now defined as the percentage of rows with the $\ell_2$-norm less than $10^{-5}$.  From Table \ref{Tab:  21PCA}, we see that our proposed StoRADA-RM consistently returns the best solutions in terms of the objective value $\Phi$ while requiring the least computational time. In particular, StoRADA-RM and R-ProxSPB can obtain solutions with similar sparsity level and loss value, while RSSM has difficulty producing a sparse solution. 
			We also plot the objective value $\Phi$ against the runtime for two representative instances of the general case in Figure \ref{Fig: 21pca}. As illustrated in the figure, our proposed StoRADA-RM is still the most efficient in reducing the objective value $\Phi$.

				\begin{figure}[h]
				\centering
				\begin{subfigure}
					\centering
					\includegraphics[scale=0.3]{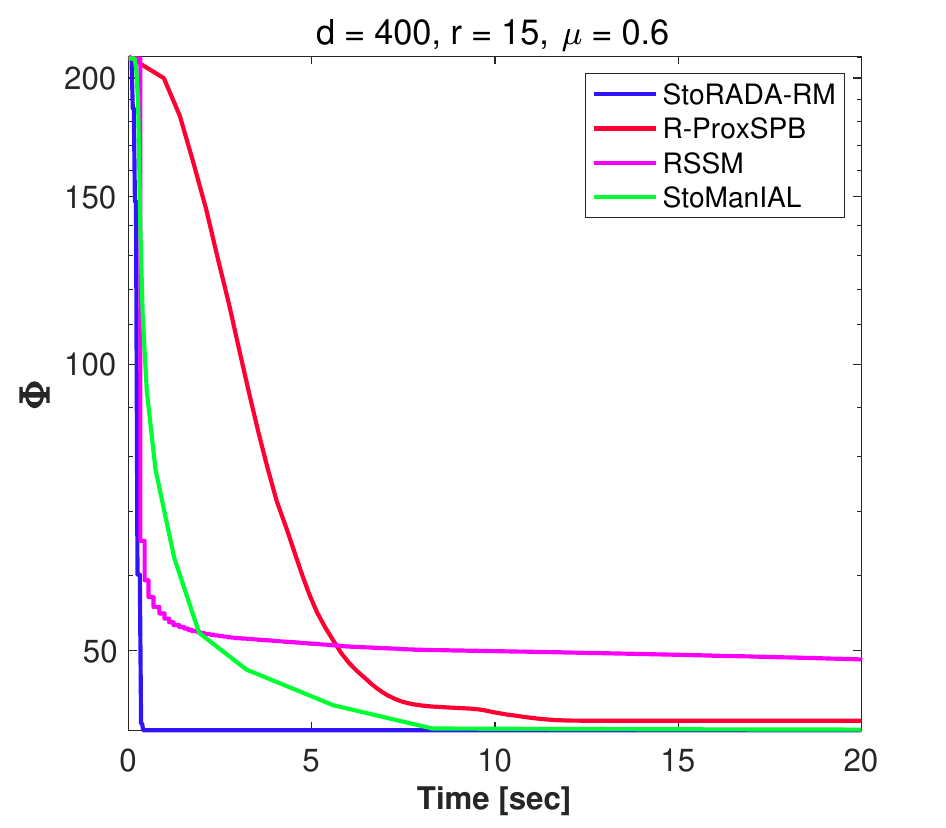}
					\end{subfigure}
					\begin{subfigure}
						\centering
						\includegraphics[scale=0.3]{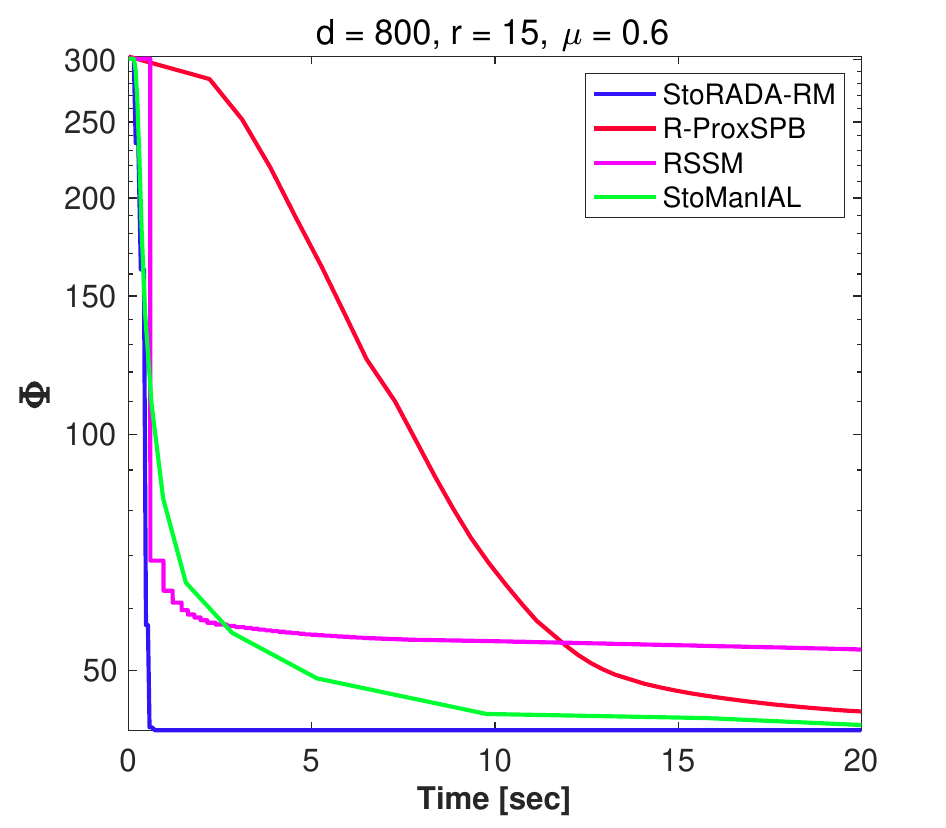}
					\end{subfigure}
					\caption{Objective value \(\Phi\) versus runtime comparison on CISE in the general case.}
					\label{Fig: 21pca}
				\end{figure}

\section{Concluding Remarks}\label{sec: concluding remarks}
In this paper, we proposed an efficient algorithm called StoRADA-RM to tackle a class of Riemannian nonsmooth composite expectation optimization problems. Our proposed method exploits the intrinsic minimax structure of these problems by performing one or multiple Riemannian stochastic gradient descent steps and then a proximal gradient ascent step at each iteration. To compute the Riemannian stochastic gradient, we proposed a vector transport-free recursive momentum estimator that requires only $\mathcal{O}(1)$ SFO calls at each iteration. Theoretically, we proved that both the iteration complexity and SFO complexity of StoRADA-RM for finding an $\epsilon$-RSS point of a given problem in the said class are $\mathcal{O}(\epsilon^{-3})$, which is the best known in the literature. The SFO complexity bound also matches the optimal lower bound for smooth nonconvex optimization with stochastic first-order algorithms.  Computationally, we presented numerical results on sparse PCA and CISE to demonstrate the superiority of StoRADA-RM over existing state-of-the-art algorithms. An interesting direction for future work is to extend the proposed algorithm to handle the setting where $f$ is also nonsmooth.

\bibliographystyle{siamplain}
\bibliography{reference_arv}

@article{li2023riemannian,
	title={A {R}iemannian alternating direction method of multipliers},
	author={Li, J. and Ma, S. and Srivastava, T.},
	journal={Math. Oper. Res.},
	volume = {50},
	number = {4},
	pages = {3222-3242},
	year={2024},
	publisher={INFORMS}
}

@article{beck2023dynamic,
	author = {Beck, A. and Rosset, I.},
	title = {A Dynamic Smoothing Technique for a Class of Nonsmooth Optimization Problems on Manifolds},
	journal = {SIAM J.  Optim.},
	volume = {33},
	number = {3},
	pages = {1473-1493},
	year = {2023}
}

@INPROCEEDINGS{xu2023efficient2,
	author={M. Xu and B. Jiang and W. Pu and Y.-F. Liu and A. M.-C. So},
	booktitle={Proc. IEEE Int. Conf. Acoust., Speech, Signal Process.}, 
	title={An Efficient Alternating {R}iemannian/Projected Gradient Descent Ascent Algorithm for Fair Principal Component Analysis}, 
	year={2024},
	volume={},
	number={},
	pages={7195-7199}
}

@article{li2021weakly,
	title={Weakly convex optimization over {S}tiefel manifold using {R}iemannian subgradient-type methods},
	author={X. Li and S. Chen and Z. Deng and Q. Qu and Z. Zhu and A. M.-C. So},
	journal={SIAM J. Optim.},
	volume={31},
	number={3},
	pages={1605--1634},
	year={2021},
	publisher={SIAM}
}

@article{wang2022manifold,
	title={A manifold proximal linear method for sparse spectral clustering with application to single-cell {RNA} sequencing data analysis},
	author={Z. Wang and B. Liu and S. Chen and S. Ma and L. Xue and H. Zhao},
	journal={INFORMS J. Optim.},
	volume={4},
	number={2},
	pages={200--214},
	year={2022},
	publisher={INFORMS}
}

@article{zhou2023semismooth,
	title={A semismooth {N}ewton based augmented {L}agrangian method for nonsmooth optimization on matrix manifolds},
	author={Y. Zhou and C. Bao and C. Ding and J. Zhu},
	journal={Math. Program.},
	volume={201},
	number={1},
	pages={1--61},
	year={2023},
	publisher={Springer}
}

@article{peng2023riemannian,
	title={{R}iemannian Smoothing Gradient Type Algorithms for Nonsmooth Optimization Problem on Compact {R}iemannian Submanifold Embedded in {E}uclidean Space},
	author={Peng, Zheng and Wu, Weihe and Hu, Jiang and Deng, Kangkang},
	journal={Appl. Math. Optim.},
	volume={88},
	number={3},
	eid={Article 85},
	year={2023},
	publisher={Springer}
}

@book{beck2017first,
	title={First-{O}rder {M}ethods in {O}ptimization},
	author={A. Beck},
	year={2017},
	publisher={SIAM}
}

@article{boumal2019global,
	title={Global rates of convergence for nonconvex optimization on manifolds},
	author={N. Boumal and P.-A. Absil and C. Cartis},
	journal={IMA J. Numer. Anal.},
	volume={39},
	number={1},
	pages={1--33},
	year={2019},
	publisher={Oxford Univ. Press}
}

@article{huang2023inexact,
	title={An inexact {R}iemannian proximal gradient method},
	author={W. Huang and K. Wei},
	journal={Comput. Optim. Appl.},
	volume={85},
	number={1},
	pages={1--32},
	year={2023},
	publisher={Springer}
}

@book{rockafellar2009variational,
	title={Variational Analysis},
	author={R. T. Rockafellar and R. J.-B. Wets},
	volume={317},
	year={2009},
	publisher={Springer Sci. Business Media}
}

@article{huang2022riemannian,
	title={Riemannian proximal gradient methods},
	author={W. Huang and K. Wei},
	journal={Math. Program.},
	volume={194},
	number={1},
	pages={371--413},
	year={2022},
	publisher={Springer}
}

@article{chen2020proximal,
	title={Proximal gradient method for nonsmooth optimization over the {S}tiefel manifold},
	author={S. Chen and S. Ma and A. M.-C. So and T. Zhang},
	journal={SIAM J. Optim.},
	volume={30},
	number={1},
	pages={210--239},
	year={2020},
	publisher={SIAM}
}

@article{borckmans2014riemannian,
	title={A {R}iemannian subgradient algorithm for economic dispatch with valve-point effect},
	author={P. B. Borckmans and S. E. Selvan and N. Boumal and P.-A. Absil},
	journal={J. Comput. Appl. Math.},
	volume={255},
	pages={848--866},
	year={2014},
	publisher={Elsevier}
}

@article{hosseini2017riemannian,
	title={A {R}iemannian gradient sampling algorithm for nonsmooth optimization on manifolds},
	author={S. Hosseini and A. Uschmajew},
	journal={SIAM J. Optim.},
	volume={27},
	number={1},
	pages={173--189},
	year={2017},
	publisher={SIAM}
}

@article{hosseini2018line,
	title={Line search algorithms for locally {L}ipschitz functions on {R}iemannian manifolds},
	author={S. Hosseini and W. Huang and R. Yousefpour},
	journal={SIAM J. Optim.},
	volume={28},
	number={1},
	pages={596--619},
	year={2018},
	publisher={SIAM}
}

@article{lai2014splitting,
	title={A splitting method for orthogonality constrained problems},
	author={R. Lai and S. Osher},
	journal={J. Sci. Comput.},
	volume={58},
	pages={431--449},
	year={2014},
	publisher={Springer}
}

@inproceedings{kovnatsky2016madmm,
	title={{MADMM}: {A} generic algorithm for non-smooth optimization on manifolds},
	author={A. Kovnatsky and K. Glashoff and M. M. Bronstein},
	booktitle={Proc. Comput. Vis. ECCV},
	pages={680--696},
	year={2016},
	organization={Springer}
}

@article{chen2024nonsmooth,
	title={Nonsmooth Optimization over the {S}tiefel Manifold and Beyond: Proximal Gradient Method and Recent Variants},
	author={S. Chen and S. Ma and A. M.-C. So and T. Zhang},
	journal={SIAM Rev.},
	volume={66},
	number={2},
	pages={319--352},
	year={2024},
	publisher={SIAM}
}

@book{absil2008optimization,
	title={Optimization Algorithms on Matrix Manifolds},
	author={P.-A. Absil and R. Mahony and R. Sepulchre},
	year={2008},
	publisher={Princeton Univ. Press}
}

@book{boumal2023introduction,
	title={An Introduction to Optimization on Smooth Manifolds},
	author={N. Boumal},
	year={2023},
	publisher={Cambridge Univ. Press}
}

@article{deng2024oracle,
	title={Oracle Complexities of Augmented {L}agrangian Methods for Nonsmooth Composite Optimization on a Compact Submanifold},
	author={K. Deng and J. Hu and J. Wu and Z. Wen},
	journal={Math. Oper. Res.},
	year={2025},
	note={doi:{\color{blue}}
	\href{https://doi.org/10.1287/moor.2024.0498}{10.1287/moor.2024.0498}},
}

@article{xiao2021exact,
	title={Exact Penalty Function for $\ell_{2,1}$ Norm Minimization over the {S}tiefel Manifold},
	author={N. Xiao and X. Liu and Y. Yuan},
	journal={SIAM J. Optim.},
	volume={31},
	number={4},
	pages={3097--3126},
	year={2021},
	publisher={SIAM}
}

@article{deng2023manifold,
	title={A Manifold Inexact Augmented {L}agrangian Method for Nonsmooth Optimization on {R}iemannian Submanifolds in {E}uclidean Space},
	author={K. Deng and Z. Peng},
	journal={IMA J. Numer. Anal.},
	volume={43},
	number={3},
	pages={1653--1684},
	year={2023},
	publisher={Oxford Univ. Press}
}

@article{jolliffe2003modified,
	title={A modified principal component technique based on the {LASSO}},
	author={Jolliffe, I. T. and Trendafilov, N. T. and Uddin, M.},
	journal={J. Comput. Graph. Stat.},
	volume={12},
	number={3},
	pages={531--547},
	year={2003},
	publisher={Taylor \& Francis}
}

@article{hu2023constraint,
	title={A constraint dissolving approach for nonsmooth optimization over the {S}tiefel manifold},
	author={Hu, X. and Xiao, N. and Liu, X. and Toh, K.-C.},
	journal={IMA J. Numer. Anal.},
	volume={44},
	number={6},
	pages={3717--3748},
	year={2024},
	publisher={OUP}
}

@article{zhang2023riemannian,
	title={A {R}iemannian smoothing steepest descent method for non-{L}ipschitz optimization on embedded submanifolds of {$\mathbb{R}^n$}},
	author={Zhang, C. and Chen, X. and Ma, S.},
	journal={Math. Oper. Res.},
	volume={49},
	number={3},
	pages={1710--1733},
	year={2023},
	publisher={INFORMS}
}

@article{zou2018selective,
	title={A selective overview of sparse principal component analysis},
	author={Zou, H. and Xue, L.},
	journal={Proc. IEEE},
	volume={106},
	number={8},
	pages={1311--1320},
	year={2018},
	publisher={IEEE}
}

@article{jiang2017vector,
	title={Vector transport-free {SVRG} with general retraction for {R}iemannian optimization: Complexity analysis and practical implementation},
	author={Jiang, B. and Ma, S. and So, A. M.-C. and Zhang, S.},
	journal={arXiv:1705.09059},
	year={2017}
}

@article{wang2022riemannian,
	title={Riemannian stochastic proximal gradient methods for nonsmooth optimization over the {S}tiefel manifold},
	author={Wang, B. and Ma, S. and Xue, L.},
	journal={J. Mach. Learn. Res.},
	volume={23},
	number={106},
	pages={1--33},
	year={2022}
}

@article{chen2010cordinate,
	author = {Chen, X. and Zou, C. and Cook, R.},
	year = {2010},
	month = {12},
	pages = {3696--3723},
	title = {Coordinate-independent sparse sufficient dimension reduction and variable selection},
	volume = {38},
	journal = {Ann. Stat.}
}

@article{deng2025single,
	title={Single-loop $\mathcal{O}(\epsilon^{-3})$ stochastic smoothing algorithms for nonsmooth {R}iemannian optimization},
	author={Deng, K. and Peng, Z. and Wu, W.},
	journal={arXiv:2505.09485},
	year={2025}
}

@inproceedings{levy2021storm,
	title={$\text{STORM}^+$: {F}ully adaptive {SGD} with Recursive Momentum for Nonconvex Optimization},
	author={Levy, K. Y. and Kavis, A. and Cevher, V.},
	booktitle={Proc. Adv. Neural Inf. Process. Syst.},
	pages={20571--20582},
	volume  = {34},
	year={2021}
}

@article{huang2021robust,
	title={Robust low-rank matrix completion via an alternating manifold proximal gradient continuation method},
	author={Huang, M. and Ma, S. and Lai, L.},
	journal={IEEE Trans. Signal Process.},
	volume={69},
	pages={2639--2652},
	year={2021},
	publisher={IEEE}
}

@article{xu2024riemannian,
	author={M. Xu and B. Jiang and Y.-F. Liu and A. M.-C. So},
	title={A {R}iemannian Alternating Descent Ascent
	Algorithmic Framework for Nonconvex-Linear
	Minimax Problems on {R}iemannian Manifolds}, 
	journal={Math. Oper. Res.},
	year={2026},
	note={doi:{\color{blue}}
	\href{https://doi.org/10.1287/moor.2025.1055
	}{10.1287/moor.2025.1055}},
}

@article{liu2024penalty,
	title={A penalty-free infeasible approach for a class of nonsmooth optimization problems over the {S}tiefel manifold},
	author={Liu, X. and Xiao, N. and Yuan, Y.-X.},
	journal={J. Sci. Comput.},
	volume={99},
	number={2},
	eid={Article 30},
	year={2024},
	publisher={Springer}
}

@article{bonnabel2013stochastic,
	title={Stochastic gradient descent on {R}iemannian manifolds},
	author={Bonnabel, S.},
	journal={IEEE Trans. Autom. Control},
	volume={58},
	number={9},
	pages={2217--2229},
	year={2013},
	publisher={IEEE}
}

@inproceedings{zhang2016first,
	title={First-order methods for geodesically convex optimization},
	author={Zhang, H. and Sra, S.},
	booktitle={Proc. Conf. Learn. Theory (COLT)},
	pages={1617--1638},
	year={2016},
	organization={PMLR}
}

@inproceedings{kasai2018riemannian,
	title={Riemannian stochastic recursive gradient algorithm},
	author={Kasai, H. and Sato, H. and Mishra, B.},
	booktitle={Proc. Int. Conf. Mach. Learn. (ICML)},
	pages={2516--2524},
	year={2018},
	organization={PMLR}
}

@inproceedings{zhou2019faster,
	title={Faster first-order methods for stochastic non-convex optimization on {R}iemannian manifolds},
	author={Zhou, P. and Yuan, X.-T. and Feng, J.},
	booktitle={Proc. Int. Conf. Artif. Intell. Stat. (AISTATS)},
	pages={138--147},
	year={2019},
	organization={PMLR}
}

@article{zhang2018r,
	title={{R}-{SPIDER}: A fast {R}iemannian stochastic optimization algorithm with curvature independent rate},
	author={Zhang, J. and Zhang, H. and Sra, S.},
	journal={arXiv:1811.04194},
	year={2018}
}

@inproceedings{han2021riemannian,
	title     = {Riemannian stochastic recursive momentum method for non-convex optimization},
	author    = {Han, A. and Gao, J.},
	booktitle = {Proc. 30th Int. Joint Conf. on Artificial Intelligence},
	pages     = {2505--2511},
	year      = {2021}
}

@article{pham2020proxsarah,
	title   = {{P}rox{SARAH}: An efficient algorithmic framework for stochastic composite nonconvex optimization},
	author  = {Pham, N. H. and Nguyen, L. M. and Phan, D. T. and Tran-Dinh, Q.},
	journal = {J. Mach. Learn. Res.},
	volume  = {21},
	number  = {110},
	pages   = {1--48},
	year    = {2020}
}

@article{rosasco2020convergence,
	title   = {Convergence of stochastic proximal gradient algorithm},
	author  = {Rosasco, L. and Villa, S. and V{\~u}, B. C.},
	journal = {Appl. Math. Optim.},
	volume  = {82},
	number  = {3},
	pages   = {891--917},
	year    = {2020},
	publisher = {Springer}
}

@inproceedings{wang2019spiderboost,
	title   = {Spider{B}oost and momentum: {F}aster variance reduction algorithms},
	author  = {Wang, Z. and Ji, K. and Zhou, Y. and Liang, Y. and Tarokh, V.},
	booktitle = {Proc. Adv. Neural Inf. Process. Syst.},
	volume  = {32},
	pages   = {2406--2416},
	year    = {2019}
}

@article{tran2022hybrid,
	title   = {A hybrid stochastic optimization framework for composite nonconvex optimization},
	author  = {Tran-Dinh, Q. and Pham, N. H. and Phan, D. T. and Nguyen, L. M.},
	journal = {Math. Program.},
	volume  = {191},
	number  = {2},
	pages   = {1005--1071},
	year    = {2022},
	publisher = {Springer}
}

@article{xu2023momentum,
	title   = {Momentum-based variance-reduced proximal stochastic gradient method for composite nonconvex stochastic optimization},
	author  = {Xu, Y. and Xu, Y.},
	journal = {J. Optim. Theory Appl.},
	volume  = {196},
	number  = {1},
	pages   = {266--297},
	year    = {2023},
	publisher = {Springer}
}

@inproceedings{cutkosky2019momentum,
	title   = {Momentum-based variance reduction in non-convex {SGD}},
	author  = {Cutkosky, A. and Orabona, F.},
	booktitle = {Proc. Adv. Neural Inf. Process. Syst.},
	volume  = {32},
	pages={15236--15245},
	year    = {2019}
}

@article{xu2025oracle,
	title   = {On the oracle complexity of a {R}iemannian inexact augmented {L}agrangian method for nonsmooth composite problems over {R}iemannian submanifolds},
	author  = {Xu, M. and Jiang, B. and Liu, Y.-F. and So, A. M.-C.},
	journal = {Optim. Lett.},
	pages   = {1--19},
	year    = {2025}
}

@article{arjevani2023lower,
	title   = {Lower bounds for non-convex stochastic optimization},
	author  = {Arjevani, Y. and Carmon, Y. and Duchi, J. C. and Foster, D. J. and Srebro, N. and Woodworth, B.},
	journal = {Math. Program.},
	volume  = {199},
	number  = {1},
	pages   = {165--214},
	year    = {2023}
}

@article{nene1996columbia,
	title={Columbia object image library ({COIL}-100)},
	author={Nene, S. A. and Nayar, S. K. and Murase, H.},
	year={1996},
	journal={Technical Report CUCS-006-96}
}

@article{sato2019riemannian,
	title={Riemannian stochastic variance reduced gradient algorithm with retraction and vector transport},
	author={Sato, H. and Kasai, H. and Mishra, B.},
	journal={SIAM J. Optim.},
	volume={29},
	number={2},
	pages={1444--1472},
	year={2019},
	publisher={SIAM}
}
\end{document}